\documentclass[a4paper,12pt]{article}
\usepackage{amsmath}
\usepackage[utf8]{inputenc}
\usepackage{fullpage}
\usepackage{amssymb}
\usepackage{amsfonts}
\usepackage{amsthm}
\usepackage{graphicx}
\usepackage{cancel}
\usepackage{float}
\usepackage{cite}
\usepackage{algorithm2e}
\usepackage{comment}
\usepackage{enumitem}
\usepackage{mathrsfs}
\newtheoremstyle{dotless}{}{}{\itshape}{}{\bfseries}{}{ }{}
\theoremstyle{dotless}
\newtheorem{theorem}{Theorem}[section]

\newtheorem{lemma}{Lemma}[section]

\theoremstyle{definition}

\theoremstyle{definition}
\newtheorem{construction}{Construction}[section]

\theoremstyle{definition}
\newtheorem{notation}{Notation}[section]

\DeclareMathOperator{\diag}{diag}

\title{Characterization of graphs attaining the maximum signless Laplacian spectral radius under forbidden cycles and theta graphs}

\author{Mainak Basunia\thanks{Department of Mathematics, Indian Institute of Technology Kharagpur, Kharagpur 721302, India. Email: leo28mynnix@gmail.com}\and Pratima Panigrahi\thanks{Department of Mathematics, Indian Institute of Technology Kharagpur, Kharagpur 721302, India. Email: pratima@maths.iitkgp.ac.in}}

\date{}
\usepackage[a4paper,margin=1in]{geometry}
\begin{document}
\maketitle
\baselineskip=0.22in

\begin{abstract}
\noindent Spectral Tur\'an-type problems ask how the absence of prescribed
subgraphs constrains the spectral radius of a matrix associated with
a graph. Given a family of graphs $\mathcal{F}$, a graph is called
$\mathcal{F}$-free if it contains no member of $\mathcal{F}$ as a
subgraph. The theta graph $\theta(l_1,\ldots,l_k)$ consists of $k$
internally disjoint paths of lengths $l_1,\ldots,l_k$ with two common end
vertices. In this paper, we study two spectral Tur\'an-type
extremal problems for the signless Laplacian spectral radius. First, among all $\{C_3,C_4\}$-free graphs of fixed order with no
pendant vertices, we determine the maximum signless Laplacian
spectral radius and uniquely characterize the extremal graph attaining it. The extremal structure
exhibits a parity phenomenon: odd and even orders give rise to two
distinct graph families. These results, in particular, sharpen a
recent general upper bound for this class given by Liu and Wang (2026). Next, we obtain the corresponding extremal results for all
$\{\theta(1,2,2),\theta(1,2,3)\}$-free graphs of fixed size with no
pendant vertices when the size is congruent to $1$ modulo $3$ and $2$ modulo
$3$, again obtaining unique but structurally different maximizing graphs in the two cases. Together with the previously known result for sizes congruent
to $0$ modulo $3$ by Liu and Wang (2026), this completes the fixed-size problem
across all three congruence classes modulo $3$.


\end{abstract}


\section{Introduction}\label{sec1}

Every graph $G$, with vertex set $V(G)$ and edge set $E(G)$, taken into consideration in this work is connected, simple, and undirected. The order and size of $G$, denoted by $n$ and $m$, are the number of vertices and edges of $G$, respectively. The neighborhood of a vertex $v \in V(G)$, denoted by $N_G(v)$, is the set of all vertices adjacent to $v$. The closed neighborhood of $v$ is $N_G[v] = N_G(v) \cup \{v\}$. The degree of the vertex $v$, denoted by $d_G(v)$, is $|N_G(v)|$. The maximum and minimum vertex degrees of $G$ are denoted by $\Delta(G)$ and $\delta(G)$, respectively. If there is no ambiguity, we refer $d_G(v)$, $\delta(G)$, $\Delta(G)$, $N_G(v)$ simply as $d(v)$, $\delta$, $\Delta$, $N(v)$, respectively. Given a set $S \subseteq V(G)$, let $G[S]$ denote the subgraph of $G$ induced by $S$, and for any $v\in V(G)$, let $N_S(v)=N(v) \cap S$, $d_S(v)=|N_S(v)|$, and $N^2(v)$ be the set of all vertices in $G$ which are at distance two from $v$. For two disjoint subsets $V_1, V_2 \subseteq V(G)$, let $e(V_1, V_2)$ denote the number of edges with one end vertex in $V_1$ and the other in $V_2$. Also, let $e(V_1)$ denote the number of edges with both end vertices in $V_1$.

For a square matrix $M$ whose eigenvalues are all real, we denote the largest eigenvalue of $M$ by $\lambda_{\max}(M)$. Moreover, the \emph{spectral radius} of $M$ is the maximum of modulus of all eigenvalues. The \emph{spectral norm} of $M$ is defined by $\lVert M\rVert=\sqrt{\lambda_{\max}(M^T M)}$. For a graph $G$, the signless Laplacian matrix of $G$ is $Q(G) = D(G) + A(G)$, where $A(G)$ is the adjacency matrix of $G$ and $D(G)$ is the diagonal matrix with the diagonal entries as vertex degrees of $G$. The spectral radius of $Q(G)$, called the \emph{signless Laplacian spectral radius} or the $Q$-index of $G$, is denoted by $q(G)$. If $G$ is connected, then $Q(G)$ is nonnegative and irreducible. Therefore, if \(G\) is connected, then, by the Perron--Frobenius theorem \cite{matrix_by_horn_johnson}, $\lambda_{\max}(Q(G))=q(G)$, and it is a simple eigenvalue and admits a unique positive unit eigenvector $X(G) = (x_1, x_2, \ldots, x_n)^T$, which is referred to as the \emph{Perron vector} of $Q(G)$.

We use the standard notations $C_n$, $P_n$, $K_n$, and $K_{a,n-a}$ to denote the cycle, path, complete graph, and complete bipartite graph of order $n$, respectively. The graph $K_{1,n-1}$ is referred to as the star graph. For positive integers $l_i$, $1 \le i \le k$ and $k\ge 3$, the \emph{theta graph} $\theta(l_1,l_2, \ldots, l_k)$ consists of $k$ internally disjoint paths of length $l_i$, which have one common end vertex $x$ and the another common end vertex $y$ ($\ne x$). The vertices $x$ and $y$ are referred to as the \emph{terminal vertices}, and the $k$ internally disjoint paths are referred to as \emph{constituent paths} of the theta graph. If for some $j \in \{1, 2, \ldots, k\}$, there are $r$ ($1\leq r \leq k$) paths of length $l_j$ in the theta graph, then we use ${l_j}^{r}$ in the notation. For example, $\theta(2,3^4,5^2)$ is the graph $\theta(2,3,3,3,3,5,5)$. When $n\geq 5$ is an odd integer, the \emph{friendship graph} $F_n$ of order $n$ is composed of $\frac{n-1}{2}$ triangles having precisely one vertex (called the \emph{central vertex}) common to all of them.

\begin{construction}
    We construct the following two graphs which will be used in our results.
\begin{enumerate}[topsep=4pt,itemsep=0pt]
\renewcommand{\labelenumi}{(\roman{enumi})}
    \item Let $P_1$ and $P_2$ be two constituent paths of length 3 in $\theta(2,3^{\,r-1})$, $r\ge 3$, with $x$ and $y$ as the terminal vertices. The graph $T_r$ of order $2r+2$ is obtained from $\theta(2,3^{\,r-1})$ by introducing a new vertex $w$ and making $w$ adjacent to the neighbor of $x$ in $P_1$ and to the neighbor of $y$ in $P_2$ (see Figure \ref{fi12}(a)).
    \item For $r\ge 4$ and odd $k \ge 5$, the graph $H_k^{(r)}$ of order $k+r-1$ is obtained from $F_k$ by identifying the central vertex of $F_k$ with a vertex of $C_r$ (see Figure \ref{fi12}(b)).
\end{enumerate}    
\end{construction}
\begin{figure}[ht]
    \centering
    \includegraphics[scale=.4]{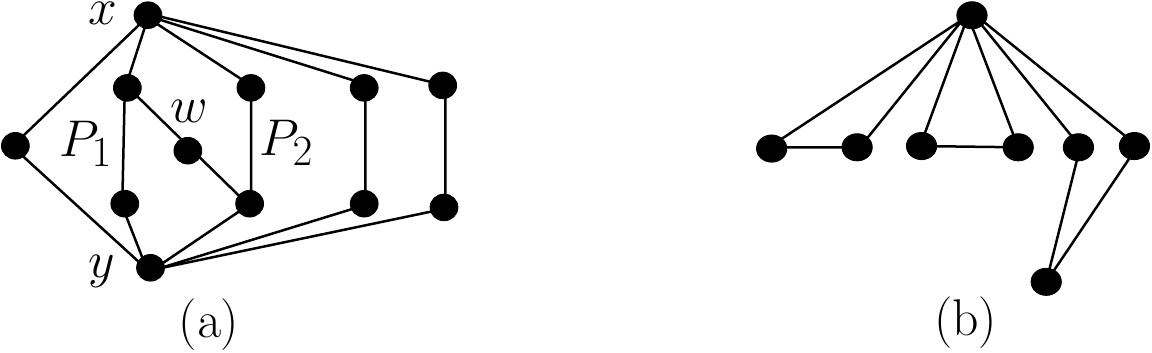}
    \caption{(a) The graph $T_5$ and (b) the graph $H_5^{(4)}$.}
    \label{fi12}
\end{figure}

Let $\mathcal{F} = \{F_1, F_2, \ldots, F_k\}$ be a family of graphs. A graph $G$ is considered $\mathcal{F}$-free if it does not contain any member of $\mathcal{F}$ as a subgraph. In particular, if $\mathcal{F} = \{F\}$, then $G$ is referred to as $F$-free. In 2010, Nikiforov~\cite{spectral_rad_without_by_nikiforov_2010} proposed the spectral Tur\'an type problem, which seeks to determine the maximum spectral radius among all $F$-free graphs of fixed order $n$. This problem is also known as the Brualdi-Solheid-Tur\'an type problem~\cite{spec_rad_by_brualdi_solheid}. It has been extensively studied for various classes of forbidden subgraphs, including complete graphs \cite{signless_2k3_by_zhang_wang, maxima_3k3_by_zhang_wang}, cycles and paths of prescribed lengths \cite{spectral_rad_without_by_nikiforov_2010}, wheels \cite{maximum_spec_rad_wheel_free_by_zhao}, intersecting cycles \cite{maximum_spec_rad_friendship_by_cioaba,spectral_no_int_odd_cycles_by_li}, and linear forests \cite{spectral_linear_forests_by_chen}. Analogous extremal problems for the signless Laplacian spectral radius are studied in \cite{maxima_Q_index_by_freitas_nikiforov_patuzzi, maxima_q_without_long_paths_nikiforov, signless_no_intersecting_triangles_by_zhao, spectral_disjoint_cliques_by_ni, alpha_by_chen, spectral_without_trees_by_hou, spectral_h2k_by_yuan}. 

Recently, Liu and Wang \cite{maxima_Q_delta_by_liu_wang} obtained an
upper bound on the signless Laplacian spectral radius of
$\{C_3,C_4\}$-free graphs of fixed order with no pendant vertices,
with equality only for the graph of order $5$.

\begin{theorem}\textup{\cite{maxima_Q_delta_by_liu_wang}}\label{r4}
Let $G$ be a $\{C_3, C_4\}$-free graph of order $n \geq 5$, with no pendant vertices. Then $q(G) \leq \frac{n+3}{2}$. Equality holds if and only if $G \cong C_5$ and $n=5$.
\end{theorem}

The restriction on pendant vertices makes the corresponding extremal
problem structurally more constrained. In particular, it rules out
the simplest star-like concentration of edges and makes the forbidden
subgraph condition play a more decisive role in determining the
maximizing graph. This naturally leads to the problem of determining
the sharp bound in Theorem~\ref{r4} and identifying the corresponding
maximizing graphs.

Another important variant of the spectral Tur\'an type problem is the Brualdi-Hoffman-Tur\'an type problem \cite{spectral_rad_by_brualdi_hoffman_1985}, where the number of edges is fixed. More precisely, this problem asks for the maximum spectral radius among all $\mathcal{F}$-free graphs of fixed size $m$. One of the earliest results in this direction is due to Nosal~\cite{eigenvalues_by_nosal_1970}, who solved the problem for triangle-free graphs. Later, Nikiforov~\cite{some_inequalities_nikiforov_2002} extended this to complete graphs. Since then, this extremal problem has been studied from several different perspectives \cite{sharp_upper_bound_on_spectral_size_by_min, sharp_upper_bound_spec_rad_theta133_by_liu, spectral_extrema_size_forbidden_fan_friendship_theta_by_li, spectral_rad_size_forbidden_by_liu_wang_2024, spectral_rad_forbidden_c7_by_lu, a_bht_problem_on_cycles_by_li, maximum_spec_rad_theta_by_gao, characterizing_by_basunia, spectral_f2_free_by_chen, maximum_theta_by_gao_2026_dam, spec_rad_turan_by_rehman_pirzada_akce, maximum_h33-even_by_pirzada, spectral_intersecting_cycle_by_rehman_pirzada}. In particular,
theta-free spectral extremal problems have attracted considerable
attention
\cite{strengthening_by_zhai,
forbidden_theta_non_bipartite_by_li_sun_2023,
maximizing_signless_theta_by_liu_wang,
maximum_alpha_spec_rad_c_by_pirzada_rehman,
extensions_by_sun_li_wei,spectral_extrema_fixed_size_by_zhai_lin_shu,
maximum_spec_rad_forbidden_by_fang}.

For the signless Laplacian spectral radius, Liu and Wang
\cite{maxima_Q_delta_by_liu_wang} determined the maximizing graph among
all $\{\theta(1,2,2),\theta(1,2,3)\}$-free graphs of size
$m=3k$, $k\geq3$, with no pendant vertices.

\begin{theorem}\textup{\cite{maxima_Q_delta_by_liu_wang}}\label{r3}
Let $G$ be a $\{\theta (1, 2, 2), \theta (1, 2, 3)\}$-free graph of size $m = 3k$, $k \geq 3$, with no pendant vertices. Then $q(G) \leq q\big(F_{\frac{2m+3}{3}}\big)$. Equality holds if and only if $G \cong F_{\frac{2m+3}{3}}$.
\end{theorem}

Motivated by Theorem~\ref{r4}, in this paper, we first consider the fixed-order
extremal problem for $\{C_3,C_4\}$-free graphs with no pendant
vertices. We determine, for each order in the considered range, the
maximum signless Laplacian spectral radius and characterize the unique
graph attaining it. The problem exhibits a
clear parity phenomenon: the odd- and even-order cases lead to two
structurally distinct extremal graph families, which we establish
in two separate theorems in Section \ref{sec3}. In particular, these results sharpen the
general upper bound given by Liu and Wang
\cite{maxima_Q_delta_by_liu_wang} in Theorem \ref{r4}.

We next consider the fixed-size problem for
$\{\theta(1,2,2),\theta(1,2,3)\}$-free graphs with no pendant
vertices. Theorem~\ref{r3} settles the case $m\equiv0\pmod 3$,
whereas the remaining two congruence classes are unresolved. We settle
these cases by determining the maximum signless Laplacian spectral
radius and characterizing the unique maximizing graphs for
$m\equiv1\pmod3$ and $m\equiv2\pmod3$, respectively, again in two
separate theorems in Section \ref{sec4}. The two congruence classes yield distinct extremal
graph families. Consequently, together with Theorem~\ref{r3}, our
results complete the fixed-size problem for all three congruence
classes modulo $3$.


\section{Preliminaries}\label{sec2}

In this section, we introduce some more notations and recall some useful lemmas. We use $I_r$, $\mathbf{1}_r$ (or simply $\mathbf{1}$) and $\mathbf{0}_{r \times s}$ (or simply $\mathbf{0}$) to denote the identity matrix of order $r$, the all-one column vector of order $r$, and the all-zero matrix of order $r\times s$, respectively. Let $M$ be an $n\times n$ real symmetric matrix, and let $\mathcal{P}: V=V_1\cup V_2\cup\cdots\cup V_k$ be a partition of the index set $V=\{1,2,\ldots,n\}$, where $1\leq k\leq n$. With respect to the partition $\mathcal{P}$, suppose that $M$ has the block representation
\begin{align*}
    M =
    \left(\begin{smallmatrix}
        M_{11} & M_{12} & \cdots & M_{1k} \\
        M_{21} & M_{22} & \cdots & M_{2k} \\
        \vdots & \vdots & \ddots & \vdots \\
        M_{k1} & M_{k2} & \cdots & M_{kk}
    \end{smallmatrix}\right),
\end{align*}
where $M_{ij}$ is the submatrix whose rows are indexed by $V_i$ and columns are indexed by $V_j$. If, for each $i,j\in\{1,2,\ldots,k\}$, every row of $M_{ij}$ has the same sum, say $b_{ij}$, then the partition $\mathcal{P}$ is called equitable, and the matrix $B = (b_{ij})_{k \times k}$ is referred to as the quotient matrix associated to $M$ with respect to the equitable partition $\mathcal{P}$.

\begin{lemma}\textup{\cite{spectra_by_brouwer_haemers,algebraic_graph_theory_by_godsil,distance_signless_laplacian_by_atik}}\label{lem4}
    Suppose $B$ is a quotient matrix corresponding to an equitable partition of a real symmetric matrix $M$. Then every eigenvalue of $B$ is also an eigenvalue of $M$. Moreover, if $M$ is irreducible and nonnegative, then the largest eigenvalues of $M$ and $B$ coincide, that is, $\lambda_{\max}(M)=\lambda_{\max}(B)$.
\end{lemma}

\begin{lemma}\label{lem8}
    Let $M$ be a real symmetric matrix of order $n$ and let $x_0$ be any real number. If $x_0 I_n -M$ is a positive definite matrix, then $x_0 > \lambda_{\max}(M)$.
\end{lemma}

\begin{proof}
    Let $\lambda_i(M)$, $i=1,2, \ldots, n$, be the eigenvalues of the matrix $M$. Then eigenvalues of $x_0 I_n -M$ are $x_0 - \lambda_i(M)$, $i=1,2, \ldots, n$, by \cite{matrix_by_horn_johnson}. Since $x_0 I_n -M$ is positive definite, it yields $x_0 - \lambda_i(M) > 0$, $i=1,2, \ldots, n$, implying $x_0 > \lambda_{\max}(M)$.
\end{proof}

\begin{lemma}\textup{(Sylvester's criterion) \cite{matrix_by_horn_johnson}}\label{lem9}
    Let $M$ be a real symmetric matrix. Then $M$ is positive definite if and only if all of its leading principal minors are positive.
\end{lemma}

\begin{lemma}\textup{\cite{geometric_by_gallier}}\label{lem10}
    Let $M$ be a real symmetric matrix of the form
    $M = \left(\begin{smallmatrix}
        A & B \\
        B^T & C
    \end{smallmatrix}\right)$.
    \begin{enumerate}[topsep=4pt,itemsep=0pt,label= {\upshape (\roman*)}]
        \item If $C$ is invertible, then $M$ is positive definite if and only if $C$ is positive definite and $A-BC^{-1}B^T$ is positive definite.
        \item If $A$ is invertible, then $M$ is positive definite if and only if $A$ is positive definite and $C-B^TA^{-1}B$ is positive definite.
    \end{enumerate}
\end{lemma}

\begin{lemma}\textup{\cite{eigenvalue_by_cvetkovic_rowlinson_simic, A_Q-merging_by_Nikiforov}}\label{lem1}
    If $G$ is a graph that contains at least one edge, then $q(G)\geq \Delta + 1$. Equality is guaranteed if and only if $G$ is a star.
\end{lemma}

\begin{lemma}\textup{\cite{on_three_conjectures_by_feng_yu}}\label{lem2}
    Consider a connected graph $G$. Then $q(G) \leq \max\big\{ d(v) +$ \linebreak $\frac{1}{d(v)}\sum_{w \in N(v)}d(w) : v \in V(G)\big\}$. Equality is satisfied only when $G$ is a semi-regular bipartite graph or a regular graph.
\end{lemma}

\begin{lemma}\textup{\cite[Theorem $2.1$]{sharp_upper_lower_by_hong_zhang}}\label{lem3}
    Let $G$ be a connected graph and let $u, v$ be two vertices in it. Suppose $v_1, v_2,\ldots, v_s \in N(v) \setminus (N(u)\cup \{u\})$, where $1 \leq s  \leq d(v)$. We construct $\widetilde{G}$ from $G$ by deleting the edges $vv_i$ and adding $uv_i$ for $i=1,2,\ldots, s$. Let the Perron vector of $Q(G)$ be denoted by $X = (x_1, x_2,\ldots, x_n)^t$, where $x_i$ corresponds to $v_i$ $(1 \leq i \leq n)$. If $x_u \geq  x_v$, then $q(\widetilde{G}) > q(G)$.
\end{lemma}

\begin{lemma}\label{lem7}
    Let $G$ be any graph of size $m$ with minimum degree $\delta\geq 2$. Then, for any vertex $u\in V(G)$, we have $2m \geq 3d(u) + 2|\overline{N[u]}|$.
\end{lemma}

\begin{proof}
    By the handshaking lemma, $2m = d(u) + \sum_{v \in V(G) \setminus \{u\}}d(v)$. Since $\delta\geq 2$, we get $2m \geq d(u) + 2(|N(u)| + |\overline{N[u]}|)$, and the desired inequality follows. 
\end{proof}


\section{Fixed-order extremal characterization for $\{C_3,C_4\}$-free graphs} \label{sec3}

Let $\mathscr{G}$ and $\mathcal{G}$ denote, respectively, the class of all $\{C_3,C_4\}$-free graphs of order $n$ with no pendant vertices and the class of all $\{\theta(1,2,2),\theta(1,2,3)\}$-free graphs of size $m$ with no pendant vertices. We first establish a lemma that will be used in the proofs of the subsequent results.

\begin{lemma}\label{lem11}
    Let $G\in \mathscr{G}$. For $v\in V(G)$, we have
    \begin{enumerate}[topsep=4pt,itemsep=0pt,label=\upshape(\roman*),ref=\upshape(\roman*)]
        \item\label{p1} $d(v) \le e(N(v),\overline{N[v]}) \le n-1-d(v)$,
        \item\label{p2} $\sum_{w\in N(v)}d(w) \le n-1$, and
        \item\label{p3} if $M(u)=\max\big\{M(v)=d(v) + \frac{1}{d(v)} \sum_{w \in N(v)} d(w) : v\in V(G)\big\}$, then $q(G) \le d(u) +\frac{n-1}{d(u)}$.
    \end{enumerate}
\end{lemma}

\begin{proof}
    \begin{enumerate}[topsep=4pt,itemsep=0pt,label=\upshape(\roman*),ref=\upshape(\roman*)]
        \item Since $G$ is $C_3$-free, $N(v)$ is independent. Also, since $G$ has no pendant vertices, every vertex of $N(v)$ must have a neighbor in $\overline{N[v]}$. Hence $e(N(v), \overline{N[v]})\geq |N(v)| = d(v)$. On the other hand, since $G$ is $C_4$-free, each vertex of $\overline{N[v]}$ has at most one
        neighbor in $N(v)$; otherwise, together with $v$, a copy of $C_4$ would arise. Therefore $e(N(v),\overline{N[v]})\leq |\overline{N[v]}| = n-1-d(v)$. By combining the two inequalities, we get the result.
        \item Since $N(v)$ is independent, $e(N(v))=0$. Using this and Lemma \ref{lem11}\ref{p1}, we get $\sum_{w \in N(v)} d(w)= d(v) + e(N(v),\overline{N[v]}) + 2e(N(v))\le d(v) + \big(n-1-d(v)\big) = n-1$.
        \item We obtain the required inequality by Lemma \ref{lem11}\ref{p2} and Lemma \ref{lem2}.\qedhere
    \end{enumerate}
\end{proof}

In the rest of the paper, we use the following notations given in Notation \ref{n1}.

\begin{notation}\label{n1}
    For any vertex $v$ in a graph $G$, $M(v)$ is given by $M(v)=d(v) + \frac{1}{d(v)} \sum_{w \in N(v)} d(w)$. We denote by $G^*$ a graph attaining the maximum signless Laplacian spectral radius in the corresponding graph class, namely $\mathscr{G}$ (or $\mathcal{G}$), and by $u$ a vertex of $G^*$ such that $M(u)=\max\{M(v):v\in V(G^*)\}$. We use $N$ for $N(u)$ and $\overline{N}$ for complement of $N[u]$ in $V(G^*)$. The Perron vector of $Q(G^*)$ is denoted by $X=X(G^*)$, with $x_v$ denoting the component of $X$ corresponding to $v\in V(G^*)$.
\end{notation}

\begin{theorem}\label{r5}
Let $G$ be a $\{C_3, C_4\}$-free graph of order $n =2k+3$, $k \geq 10$, with no pendant vertex. Then $q(G) \leq q'(n)$, where $q'(n)$ is the largest root of $x^3 -\frac{1}{2}(n+9)x^2 + (2n+3)x -2n+6=0$. Moreover, equality holds if and only if $G \cong\theta (2, 3^{\frac{n-3}{2}})$.
\end{theorem}

\begin{proof}
Let $G^*$ be a graph with the maximum signless Laplacian spectral radius in the class $\mathscr{G}$ with $n=2k+3$, $k\geq 10$. The graph $\theta(2,3^{\frac{n-3}{2}})=\theta(2,3^k)$ belongs to $\mathscr{G}$ and has maximum degree $k+1$. With respect to the partition consisting of two terminal vertices, the internal vertex on the path of length $2$, and the internal vertices on the paths of length $3$, the quotient matrix associated with $Q(\theta(2,3^k))$ is $E_{\theta(2, 3^k)} =
    \left(\begin{smallmatrix}
        k+1 & 1 & k \\
        2 & 2 & 0 \\
        1 & 0 & 3
    \end{smallmatrix}\right)$. The characteristic polynomial of $E_{\theta(2, 3^k)}$ is $j(x)=x^3-(k+6)x^2 + (4k+9)x -4k$. If $\lambda_j$ is the largest root of $j(x)=0$, then $q(\theta(2, 3^k)) = \lambda_j$ from Lemma \ref{lem4}. Moreover, since $\theta(2,3^k)$ is not a star, Lemma~\ref{lem1} gives $q(\theta(2, 3^k)) > k+2$. Now, if possible, let $G^*\not\cong \theta(2,3^k)$. Then $q(G^*) \geq q(\theta(2, 3^k)) > k+2$. Applying Lemma \ref{lem11}\ref{p1} on $G^*$ and $u$, we get $d(u) \le 2k+2-d(u)$, which implies $d(u)\leq k+1$. Since $G^*$ does not have any pendant vertices, it follows that $2\leq d(u) \leq k+1$. We now divide the proof into cases according to the possible values of $d(u)$.
\medskip\\ 
\noindent \textbf{Case 1.} $d(u)=k+1$. Since $G^*$ is $C_3$-free, $N$ is independent in $G^*$, and $|N|=|\overline{N}|=k+1$. By Lemma \ref{lem11}\ref{p1}, we get $k+1 \le e(N,\overline{N})\le (2k+3)-1-(k+1)$, implying $e(N,\overline{N})=k+1$. Therefore, each vertex of $N$ has exactly one neighbor in $\overline{N}$, and each vertex of $\overline{N}$ has exactly one neighbor in $N$. Thus we may write $N=\{v_1,\ldots,v_{k+1}\}$, $\overline{N}=\{w_1,\ldots,w_{k+1}\}$, where $v_iw_i\in E(G^*)$ for each $i=1,2,\ldots, k+1$, and these are precisely the edges between $N$ and $\overline{N}$ in $G^*$. Let $H=G^*[\overline{N}]$. Since each $w_i$ has exactly one neighbor in $N$ and $\delta(G^*)\geq 2$, each $w_i$ must have at least one neighbor in $\overline{N}$. Hence $\delta(H)\geq 1$. Moreover, since $G^*$ is $\{C_3,C_4\}$-free, the induced subgraph $H$ is also $\{C_3,C_4\}$-free. Thus $G^*$ has the form $S(H)$, where $S(H)$ is obtained from $H$ by attaching one new vertex $v_i$ to each vertex $w_i$ and joining all the vertices $v_i$ to the vertex $u$. 

If $H\cong K_{1,k}$, then $G^*\cong \theta(2,3^k)$, contradicting the assumption $G^*\not\cong \theta(2,3^k)$. Hence we may assume that $H \not\cong K_{1,k}$. We show that this remaining possibility also leads to a contradiction. Since $H$ is $\{C_3,C_4\}$-free and $|V(H)|=k+1$, $\Delta(H)\leq k-1$. We already have $\delta(H)\geq 1$, thus for any vertex $w\in V(H)$, $1 \le d_H(w) \le k-1$. If $d_H(w)=1$, then the unique neighbor of $w$ in $H$ has degree at most $k-1$, and thus, $d_H(w)+\frac{1}{d_H(w)} \sum_{w'\in N_H(w)}d_H(w') \le 1+k-1=k$. For $2\le d_H(w) \le k-1$, applying Lemma \ref{lem11}\ref{p2} on $H$ and $w$, we get $\sum_{w'\in N_H(w)}d_H(w')\le k$, and hence $d_H(w)+\frac{1}{d_H(w)}\sum_{w'\in N_H(w)}d_H(w') \le d_H(w)+\frac{k}{d_H(w)}$. The function $f(x)=x+\tfrac{k}{x}$ is convex for $x>0$, so its maximum on the interval $[2,k-1]$ is attained at an endpoint. Thus $d_H(w)+\frac{k}{d_H(w)} \le \max\left\{2+\frac{k}{2}, k-1+\frac{k}{k-1}\right\}$. Since $\left(k-1+\tfrac{k}{k-1}\right) - \left(2+\tfrac{k}{2}\right) = \tfrac{(k-2)(k-3)}{2(k-1)} > 0$ for $k\ge 10$, $d_H(w)+\frac{k}{d_H(w)} \le k-1+\frac{k}{k-1} = k+\frac1{k-1}$. Since $\frac{1}{k-1}>0$ for $k\ge 10$, for every vertex $w\in V(H)$, $d_H(w)+\frac{1}{d_H(w)} \sum_{w'\in N_H(w)}d_H(w') \le k+\frac1{k-1}$. Applying Lemma~\ref{lem2} to each connected component of $H$, and using this inequality, we obtain $q(H)\le k+\frac1{k-1}$. Now, with respect to the partition $V(G^*)=\{u\}\cup N\cup \overline{N}$, $Q(G^*)$ has the following block representation $Q(G^*)=
\left(\begin{smallmatrix}
k+1 & \mathbf{1}^{T} & \mathbf{0}^T\\
\mathbf{1} & 2I_{k+1} & I_{k+1}\\
\mathbf{0} & I_{k+1} & Q(H)+I_{k+1}
\end{smallmatrix}\right)$. Let $z=
\left(\begin{smallmatrix}
z_1\\
z_2\\
z_3
\end{smallmatrix}\right)$,
where $z_1\in \mathbb{R}$, $z_2\in \mathbb{R}^{k+1}$ and $z_3\in \mathbb{R}^{k+1}$.
Then $z^TQ(G^*)z
=
(k+1){z_1}^2
+
2z_1\mathbf{1}^Tz_2
+
2\lVert z_2 \rVert^2
+
2z_2^Tz_3
+
z_3^TQ(H)z_3 + \lVert z_3 \rVert^2$. By Rayleigh quotient, we have $z_3^TQ(H)z_3\le q(H)\lVert z_3 \rVert^2$. Furthermore, applying Cauchy-Schwarz inequality, $z^TQ(G^*)z
\le
(k+1)z_1^2
+
2\sqrt{k+1}|z_1|\lVert z_2 \rVert
+
2\lVert z_2\rVert^2
+
2\lVert z_2\rVert \lVert z_3\rVert
+
\big(q(H)+1\big)\lVert z_3\rVert^2$. Using the inequality $q(H)\le k+\frac1{k-1}$ we just obtained, this becomes $z^TQ(G^*)z
\le
(k+1)z_1^2
+
2\sqrt{k+1}|z_1|\lVert z_2 \rVert
+
2\lVert z_2\rVert^2
+
2\lVert z_2\rVert \lVert z_3\rVert
+
\left(k+1+\frac{1}{k-1}\right)\lVert z_3\rVert^2$. We define $\widehat z=
\left(\begin{smallmatrix}
|z_1|\\
\lVert z_2\rVert\\
\lVert z_3\rVert
\end{smallmatrix}\right)$
and $L=
\left(\begin{smallmatrix}
k+1 & \sqrt{k+1} & 0\\
\sqrt{k+1} & 2 & 1\\
0 & 1 & k+1+\frac{1}{k-1}
\end{smallmatrix}\right)$.
Then the right-hand side of the last inequality is precisely $\widehat z^{T}L\widehat z$, thus, $z^TQ(G^*)z\le \widehat z^T L\widehat z$. Moreover, $\lVert\widehat z\rVert^2=|z_1|^2+\lVert z_2\rVert^2+\lVert z_3\rVert^2=\lVert z\rVert^2$. So $\lVert z\rVert =1$ if and only if $\lVert \widehat z\rVert =1$. When $\lVert \widehat z \rVert =1$, from Rayleigh quotient, we have $\widehat z^T L\widehat z\le \lambda_{\max}(L)$. Therefore, for every unit vector $z$, $z^TQ(G^*)z\le \lambda_{\max}(L)$. Taking the maximum over all unit vectors $z$, it follows that $q(G^*)=\max_{\lVert z\rVert=1}z^TQ(G^*)z\le \lambda_{\max}(L)$. Let $x_0=k+2+\frac{3}{k+1}$. Then $x_0I-L=
\left(\begin{smallmatrix}
1+\frac{3}{k+1} & -\sqrt{k+1} & 0\\
-\sqrt{k+1} & k+\frac{3}{k+1} & -1\\
0 & -1 & 1+\frac{3}{k+1}-\frac1{k-1}
\end{smallmatrix}\right)$. The first leading principal minor of this matrix is $D_1=1+\frac3{k+1}$, the second one is $D_2=
\left|\begin{smallmatrix}
1+\frac3{k+1} & -\sqrt{k+1}\\
-\sqrt{k+1} & k+\frac3{k+1}
\end{smallmatrix}\right|
=\frac{2k^2+4k+11}{(k+1)^2}$, and the third one is $D_3=\det(x_0I-L) = \frac{k^4+3k^3+6k^2+7k-51}{(k+1)^3(k-1)}$. When $k\ge 10$, $D_1, D_2, D_3>0$. Consequently, $x_0I-L$ is positive definite for $k\ge 10$ by Lemma \ref{lem9}. We further apply Lemma \ref{lem8} and obtain $\lambda_{\max}(L)<x_0=k+2+\frac{3}{k+1}$. Combining it with $q(G^*)\le \lambda_{\max}(L)$, we get $q(G^*)<k+2+\frac{3}{k+1}$. On the other hand, recalling $j(x)$, the characteristic polynomial of $E_{\theta(2, 3^k)}$, we obtain $j(k+2+\frac{3}{k+1}) = -\frac{k^3-9k^2-3k-20}{(k+1)^3}<0$ for $k\geq 10$. Since $q(\theta(2,3^k))$ is the largest root of $j(x)=0$, $q(\theta(2, 3^k))>k+2+\frac{3}{k+1}$. Together with $q(G^*)<k+2+\frac{3}{k+1}$, this gives $q(G^*)<q(\theta(2, 3^k))$, which contradicts the maximality of $q(G^*)$ in $\mathscr{G}$. Thus the possibility $H\not\cong K_{1,k}$ also leads to a contradiction. Consequently, Case 1 cannot occur under the initial assumption $G^*\not\cong \theta(2,3^k)$.
\medskip\\ 
\noindent \textbf{Case 2.} $d(u)=k$. By Lemma \ref{lem11}\ref{p3}, we get $q(G^*) \le k +\frac{2k+2}{k}=k+2+\frac{2}{k}$. On the other hand, $j(k+2+\frac{2}{k})=-\frac{k^4-2k^3-2k^2-8}{k^3}<0$, since $k^4-2k^3-2k^2-8 = k^2((k-1)^2-1)-8 > 0$ for $k\ge 10$. Thus $k+2+\frac{2}{k} < q(\theta(2,3^k))$. Hence, $q(G^*) < q(\theta(2,3^k))$, contradicting the maximality of $q(G^*)$ in $\mathscr{G}$. Therefore, Case 2 cannot occur.
\medskip\\ 
\noindent \textbf{Case 3.} $3\le d(u)\le k-1$. By Lemma \ref{lem11}\ref{p3}, $q(G^*) \le d(u) +\frac{2k+2}{d(u)}$. Due to the convexity of the function $f(x)=x+\frac{2k+2}{x}$ on the interval $[3, k-1]$, we have $q(G^*) \le \max \left\{3+\frac{2k+2}{3}, k-1 +\frac{2k+2}{k-1} \right\}$. Now $3+\frac{2k+2}{3}=\frac{2k+11}{3}\le k+2$ for $k\ge 10$. Also, $k-1 +\frac{2k+2}{k-1}\le k-1+3=k+2$, since $\frac{2k+2}{k-1}\le 3$ for $k\ge 10$. Thus $q(G^*) \le k+2$, which is a contradiction, since we already have $q(G^*) \geq q(\theta(2, 3^k)) > k+2$. This rules out Case 3.
\medskip\\ 
\noindent \textbf{Case 4.} $d(u)=2$. Let $N=\{a,b\}$. Since $G^*$ is $C_3$-free, we have $ab\notin E(G^*)$. Define $A=N(a)\setminus\{u\}$, $B=N(b)\setminus\{u\}$ and $N'=V(G^*)\setminus\big(\{u,a,b\}\cup A\cup B\big)$. Since $G^*$ is $C_4$-free, no vertex except $u$ can be adjacent to both $a$ and $b$; otherwise, such a vertex together with $a,u,b$ would form a $C_4$. Hence $A\cap B=\emptyset$. Let $|A|=r$, $|B|=s$ and $|N'|=t$. Then $r+s+t=n-3=2k$. Also, $d(a)=r+1$ and $d(b)=s+1$. Therefore $M(u) = d(u) + \frac{1}{d(u)} \sum_{v \in N} d(v) = 2+\frac{d(a)+d(b)}{2} = 2+\frac{r+s+2}{2}=2+\frac{2k-t+2}{2} = k+3-\frac{t}{2}$. We now split the case according to the possible values of $t$.
\medskip\\ 
\noindent \textbf{Subcase 4.1.} $t\ge 2$. Then $M(u)=k+3-\frac t2\le k+2$. Thus, from Lemma \ref{lem2}, $q(G^*)\le M(u)\le k+2$, a contradiction, since $q(G^*)>k+2$ already.
\medskip\\
\noindent \textbf{Subcase 4.2.} $t=1$. Since $d(a)=r+1$, $\delta(G^*)\geq 2$ and $d(u)=2$, $\sum_{x\in N(a)}d(x) = d(u) + \sum_{x\in A}d(x)\ge 2+2|A|=2(r+1)$. Hence $M(a)= d(a)+\frac{1}{d(a)}\sum_{x\in N(a)}d(x) \ge r+1+\frac{2(r+1)}{r+1} = r+3$. Also, $r+s=2k-t=2k-1$. Therefore, one of $r$ and $s$ is at least $k$. Without loss of generality, assume that $r\ge k$. Thus $M(a)\geq k+3$. On the other hand, $M(u)=k+3-\frac{t}{2}$ gives $M(u)=k+3-\frac{1}{2}=k+\frac{5}{2}$. This gives $M(a)>M(u)$, a contradiction according to the choice of $u$ in $G^*$. Therefore, the case $t=1$ cannot occur.\medskip\\\noindent \textbf{Subcase 4.3.} $t=0$. Then $r+s=2k-t=2k$ and $M(u)=k+3-\frac{t}{2}=k+3$. Since $t=0$ implies $N'=\emptyset$, every vertex other than $u,a,b$ lies in $A\cup B$. Because $G^*$ is $C_3$-free, there are no edges inside $A$ and no edges inside $B$. Moreover, since $G^*$ is $C_4$-free, each vertex of $A$ has at most one neighbor in $B$, and each vertex of $B$ has at most one neighbor in $A$. Now, every vertex of $A$ is adjacent to $a$, and since $\delta(G^*)\geq 2$, it must have at least one additional neighbor. This neighbor must lie in $B$. Similarly, every vertex of $B$ must have a neighbor in $A$. Hence the edges between $A$ and $B$ form a perfect matching, and $r=s=k$. Therefore, $G^*$ consists of the path $a-u-b$ and $k$ internally disjoint paths of length $3$ between $a$ and $b$. Hence $G^* \cong \theta(2,3^k)$, contradicting the assumption $G^*\not\cong \theta(2,3^k)$. Thus the case $t=0$ is also excluded. This completes Case 4.

We have now shown that, under the assumption $G^*\not\cong \theta(2,3^k)$, every possible value of $d(u)$ leads to a contradiction. Therefore, this assumption is false, and hence $G^*\cong \theta(2,3^k)=\theta(2,3^{\frac{n-3}{2}})$. Substituting $k=\frac{n-3}{2}$ in the expression of $j(x)$, the maximum signless Laplacian spectral radius is the largest root of $x^3-\frac{1}{2}(n+9)x^2+(2n+3)x-2n+6=0$. This proves Theorem~\ref{r5}. \qedhere
\end{proof}


\begin{theorem}\label{r6}
Let $G$ be a $\{C_3, C_4\}$-free graph of order $n =2k+2$, $k \geq 22$, with no pendant vertices. Then $q(G) \leq q'(n)$, where $q'(n)$ is the largest root of $x^5-\frac{1}{2}(n+18)x^4 + \frac{3}{2}(3n+16)x^3 - \frac{1}{2}(27n+20)x^2 + 4(4n-9)x -6n+24=0$. Moreover, equality holds if and only if $G \cong T_{\frac{n-2}{2}}$.
\end{theorem}

\begin{proof}
    Let $G^*$ be a graph with the maximum signless Laplacian spectral radius in the class $\mathscr{G}$ with $n=2k+2$, $k\geq 22$. The graph $T_{\frac{n-2}{2}}=T_k$ belongs to $\mathscr{G}$ and has maximum degree $k$. $Q(T_k)$ admits an equitable partition with corresponding quotient matrix as
\begin{align*}
    E_{T_k} =
    \left(\begin{smallmatrix}
k & 1 & 0 & 1 & 1 & k-3\\
2 & 2 & 0 & 0 & 0 & 0\\
0 & 0 & 2 & 2 & 0 & 0\\
1 & 0 & 1 & 3 & 1 & 0\\
1 & 0 & 0 & 1 & 2 & 0\\
1 & 0 & 0 & 0 & 0 & 3
\end{smallmatrix}\right).
\end{align*}
The characteristic polynomial of $E_{T_k}$ is $j_1(x)=(x-2)j(x)$, where $j(x) = x^5-(k+10)x^4+(9k+33)x^3-(27k+37)x^2+(32k-4)x-12k+12$. Since $T_k$ is not a star, Lemma~\ref{lem1} gives $q(T_k) > k+1>2$ for $k\geq 22$. Hence, by Lemma~\ref{lem4}, $q(T_k)$ is the largest root of $j(x)=0$. It follows that $j\big(k+1+\tfrac{3}{k}\big) = -\frac{1}{k^5}\big(2k^7-14k^6+65k^5-147k^4+297k^3-405k^2+405k-243\big)= -\frac{1}{k^5}\big(2k^6(k-7)+k^4(65k-147)+k^2(297k-405)+(405k-243)\big)<0$, because in the numerator, every grouped term is positive for $k\ge 22$. Therefore $q(T_k) > k+1+\frac{3}{k}$. Now, if possible, let $G^*\not\cong T_k$. Then
\begin{align}\label{eq50}
    q(G^*) \geq q(T_k) > k+1+\frac{3}{k}.
\end{align} 

\noindent By Lemma \ref{lem11}\ref{p1}, we have $d(u) \le (2k+2)-1-d(u)$, which implies $d(u) \leq k+\frac12$. Since $d(u)$ is an integer, we can write $d(u)\leq k$. Together with $\delta(G^*)\geq 2$, this gives $2\leq d(u) \leq k$. We now divide the proof into cases according to the possible values of $d(u)$.
\medskip\\ 
\noindent \textbf{Case 1.} $d(u)=k$. Then $|N|=k$ and $|\overline{N}|=k+1$. By Lemma \ref{lem11}\ref{p1}, $k\le e(N,\overline{N})\le k+1$, so either $e(N,\overline{N})=k+1$ or $e(N,\overline{N})=k$. 
\medskip\\ 
\noindent \textbf{Subcase 1.1.} $e(N,\overline{N}) = k+1$. Then every vertex of $\overline{N}$ has exactly one neighbor in $N$. Since $|N|=k$ and $|\overline{N}|=k+1$, exactly one vertex of $N$ has two neighbors in $\overline{N}$, while every other vertex of $N$ has exactly one. We call a set of vertices in $\overline{N}$ having the same unique neighbor in $N$ a \textit{fibre}. Thus $\overline{N}$ is partitioned into $k$ fibres, one of size $2$ and the remaining ones of size $1$. Let $H=G^*[\overline{N}]$. Since every vertex of $\overline{N}$ has exactly one neighbor in $N$ and $\delta(G^*)\geq 2$, we have $\delta(H)\geq 1$. Two vertices in the same fibre can neither be adjacent nor have a common neighbor in $H$, as this would create a $C_3$ or a $C_4$, respectively. Hence $\Delta(H)\neq k$, because a vertex of degree $k$ would either be adjacent to both vertices of the double fibre or lie in it and be adjacent to the other vertex. We therefore distinguish the following two possibilities.

\textbf{Possibility 1.} $\Delta(H) =k-1$. Let $w$ be a vertex of degree $k-1$ in $H$. Then $w$ cannot lie in the double fibre. Because if $w$ lies in the double fibre, then the other vertex in that double fibre still needs a neighbor in $H$. But any such neighbor is already adjacent to $w$, producing a $C_4$ through the common $N$-neighbor of the double fibre. Therefore $w$ lies in a singleton fibre. Moreover, it will be adjacent to all vertices from singleton fibres, and exactly one vertex from the double fibre. Let the other non-adjacent vertex of the double fibre be $w'$. Since $\delta(H)\ge 1$, $w'$ must be adjacent at least one vertex of $V(H)\setminus \{w\}$. Also $w'$ cannot be adjacent to two neighbors of $w$, otherwise we get a $C_4$ in $H$. Hence $w'$ is adjacent to exactly one neighbor of $w$ in $H$. This forces $G^*$ to coincide with the graph none other than $T_k$, contradicting the assumption $G^*\not\cong T_k$. Thus Possibility 1 cannot occur.

\textbf{Possibility 2.} $\Delta(H) \le k-2$. Since $H$ is also $\{C_3,C_4\}$-free, applying Lemma \ref{lem11}\ref{p2}, for every $w\in V(H)$, we get $\sum_{w'\in N_H(w)} d_H(w') \le |V(H)|-1=k$. If $d_H(w)=1$, then the unique neighbor of $w$ has degree at most $k-2$, so $d_H(w)+\frac{1}{d_H(w)}\sum_{w'\in N_H(w)}d_H(w') \le 1+(k-2)=k-1$. If $2\le d_H(w)\le k-2$, then $d_H(w)+\frac{1}{d_H(w)}\sum_{w'\in N_H(w)}d_H(w') \le d_H(w)+\frac{k}{d_H(w)}$. The function $f(x)=x+\frac{k}{x}$ is convex for $x>0$, and therefore its maximum on $[2,k-2]$ is attained at an endpoint. Since $\left(k-2+\frac{k}{k-2}\right)-\left(2+\frac{k}{2}\right)
=
\frac{k^2-7k+8}{2(k-2)}>0$ for $k\geq 22$, we obtain $d_H(w)+\frac{k}{d_H(w)} \leq k-2+\frac{k}{k-2} = k-1+\frac{2}{k-2}$. Thus, for every vertex $w\in V(H)$, $d_H(w)+\frac{1}{d_H(w)} \sum_{w'\in N_H(w)}d_H(w') \le k-1+\frac{2}{k-2}$. Applying Lemma~\ref{lem2} componentwise to $H$, we get $q(H)\leq k-1+\frac{2}{k-2}$. Now split $N=\{v_0\}\cup N'$, where $v_0$ is the unique vertex of $N$ with two neighbors in $\overline{N}$. With respect to the partition $V(G^*)=\{u\}\cup \{v_0\}\cup N'\cup \overline{N}$, $Q(G^*)$ has the following block form
\begin{align*}
    Q(G^*)=
\left(\begin{smallmatrix}
k & 1 & \mathbf{1}^T & 0\\
1 & 3 & \mathbf{0} & p^T\\
\mathbf{1} & \mathbf{0} & 2I_{k-1} & R\\
0 & p & R^T & Q(H)+I_{k+1}
\end{smallmatrix}\right).
\end{align*}
Here $p \in \mathbb{R}^{k+1}$ corresponds to the adjacency between $\overline{N}$ and $\{v_0\}$. Thus $p$ has exactly two entries equal to $1$, and hence $\lVert p\rVert=\sqrt{2}$. Also, $R\in\mathbb{R}^{(k-1)\times(k+1)}$ corresponds to the adjacency between $N'$ and $\overline{N}$. Each row of $R$ has exactly one entry equal to $1$, and different rows have their $1$'s in different columns. Hence $RR^T=I_{k-1}$ and $\lVert R\rVert=1$. Now, let $z=\left(\begin{smallmatrix}z_1\\z_2\\z_3\\z_4\end{smallmatrix}\right)$, where $z_1,z_2\in \mathbb{R}$, $z_3\in\mathbb{R}^{k-1}$, $z_4\in\mathbb{R}^{k+1}$. Using the upper bound for $q(H)$ along with the other estimates we just obtained and the Rayleigh quotient comparison, we obtain $z^TQ(G^*)z \le kz_1^2 +2|z_1||z_2| +2\sqrt{k-1}|z_1|\lVert z_3\rVert +3z_2^2 +2\lVert z_3\rVert^2 +2\sqrt2 |z_2|\lVert z_4\rVert +2\lVert z_3\rVert \lVert z_4\rVert +\Big(k+\frac{2}{k-2}\Big)\lVert z_4\rVert^2$. Setting $\widehat z=\left(\begin{smallmatrix}|z_1|\\|z_2|\\ \lVert z_3\rVert\\ \lVert z_4\rVert \end{smallmatrix}\right)$, it becomes $z^TQ(G^*)z\leq \widehat z^TL\widehat z$, where
\begin{align*}
    L=\left(\begin{smallmatrix}
k & 1 & \sqrt{k-1} & 0\\
1 & 3 & 0 & \sqrt2\\
\sqrt{k-1} & 0 & 2 & 1\\
0 & \sqrt2 & 1 & k+\frac{2}{k-2}
\end{smallmatrix}\right).
\end{align*}
Since $\lVert \widehat z\rVert=\lVert z\rVert$, it follows that $q(G^*)\le \lambda_{\max}(L)$. We show that $\lambda_{\max}(L)<k+1+\frac{3}{k}$. The leading principal minors of $\left(k+1+\frac{3}{k}\right)I-L$ are: $D_1=1+\frac{3}{k}>0$, $D_2=\frac{k^3-3k+9}{k^2}>0$ and $D_3= \frac{2k^4-5k^3+15k^2-18k+27}{k^3}>0$ for $k\ge 22$. Moreover, $D_4=\det\left(\left(k+1+\frac{3}{k}\right)I-L\right) = \frac{N(k)}{k^4(k-2)}$, where $N(k) = k^6-11k^5+23k^4-36k^3+9k^2+81k-162-2\sqrt2\,k^4(k-2)\sqrt{k-1}$. For $k\geq 22$, we have $\sqrt{k-1}\leq \frac{5k}{24}$ and $\sqrt{2}<\frac{17}{12}$. Hence $N(k) > \frac{1}{144}(59k^6-1414k^5+3312k^4-5184k^3+1296k^2+11664k-23328)>0$ for $k\ge 22$. Therefore $D_4 > 0$. By Lemma \ref{lem9}, $(k+1+\frac{3}{k})I -L$ is positive definite. Thus, by Lemma~\ref{lem8}, $\lambda_{\max}(L)<k+1+\frac{3}{k}$. Combining this with $q(G^*)\le \lambda_{\max}(L)$, we obtain $q(G^*)< k+1+\frac{3}{k}$, contrary to \eqref{eq50}. Hence Possibility~2 also leads to a contradiction, and therefore Subcase~1.1 cannot occur.\medskip\\ 
\noindent \textbf{Subcase 1.2.} $e(N,\overline{N}) = k$. Then every vertex of $N$ has exactly one neighbor in $\overline{N}$. Moreover, exactly one vertex of $\overline{N}$, say $w_0$, has no neighbor in $N$, while every other vertex of $\overline{N}$ has exactly one neighbor in $N$. Let $H=G^*[\overline{N}]$. Since $w_0$ has no neighbor in $N$ and $\delta(G^*)\geq 2$, we have $d_H(w_0)\geq 2$. Also, every vertex of $\overline{N}\setminus\{w_0\}$ has degree at least $1$ in $H$. We distinguish two possibilities according to $\Delta(H)$.

\textbf{Possibility 1.} $\Delta(H)=k$. Since $|V(H)|=k+1$, it follows that $H\cong K_{1,k}$. The center of this star must be the exceptional vertex $w_0$, for otherwise a vertex of $\overline{N}\setminus\{w_0\}$ would have degree $k+1$ in $G^*$, contradicting $\Delta(G^*)\leq k$ obtained from Lemma \ref{lem11}\ref{p1}. Hence $G^*\cong \theta(3^k)$. Now, an equitable quotient matrix of $Q(G^*)$ is $E_{G^*}=
\left(\begin{smallmatrix}
k & k\\
1 & 3
\end{smallmatrix}\right)$. Thus the characteristic polynomial of $E_{G^*}$ is $g(x)=x^2-(k+3)x+2k$, and hence, by Lemma~\ref{lem4}, $q(G^*) = \frac{k+3+\sqrt{k^2-2k+9}}{2}$. It follows that $g(k+1+\frac3k)=\frac{k^2-3k+9}{k^2}>0$ for $k\ge 22$. Also, $k+1+\frac3k-\frac{k+3}{2} =\frac{k^2-k+6}{2k} >0$, so $k+1+\frac{3}{k}$ lies to the right of the midpoint of the two roots of $g(x)=0$. Since $g$ is a monic quadratic, it follows that $k+1+\frac{3}{k}$ lies to the right of the larger root. Therefore $q(G^*)< k+1+\frac3k$, which contradicts \eqref{eq50}. This rules out Possibility 1.

\textbf{Possibility 2.} $\Delta(H) = k-1$. We choose $c\in V(H)$ such that $d_H(c)=k-1$. Since $|\overline{N}|=k+1$, there is exactly one vertex of $H$ not adjacent to $c$; we denote it by $b$. Thus $V(H)=\{b, c\}\cup N_H(c)$, where $|N_H(c)|=k-1$. Since $H$ is triangle-free, $N_H(c)$ is independent. Also, since $H$ is $C_4$-free, the vertex $b$ can have at most one neighbor in $N_H(c)$. We first observe that $b\neq w_0$. Indeed, if $b=w_0$, then $b$ has no neighbor in $N$, so $\delta(G^*)\geq 2$ would imply $d_H(b)\geq 2$, which is impossible because $b$ is not adjacent to $c$ and has at most one neighbor in $N_H(c)$. Hence $b$ has exactly one neighbor in $N$. Since $\delta(G^*)\geq 2$, the vertex $b$ must have a neighbor in $H$. Therefore, $b$ has exactly one neighbor in $N_H(c)$; we denote this vertex by $a$. Consequently, $E(H)=\{cw':w'\in N_H(c)\}\cup\{ab\}$, one of these $w'$s being the vertex $a$. Thus $H$ is a star $K_{1,k-1}$ with one extra pendant edge attached to one of its leaves. Since $w_0$ has no neighbor in $N$, we have $d_H(w_0)\geq 2$. The only vertices of $H$ with $H$-degree at least $2$ are $a$ and $c$. Hence $w_0\in\{a,c\}$.

First suppose that $w_0=a$. Then $G^*$ is precisely $\theta(2,3^{k-2},4)$. The two common end-vertices in $k$ disjoint paths are $u$ and $c$. With $G^*\cong\theta(2,3^{k-2},4)$, $Q(G^*)$ admits an equitable partition with quotient matrix
\begin{align*}
    E_{G^*}=\left(\begin{smallmatrix}
k & 1 & k-2 & 1 & 0\\
2 & 2 & 0 & 0 & 0\\
1 & 0 & 3 & 0 & 0\\
1 & 0 & 0 & 2 & 1\\
0 & 0 & 0 & 2 & 2
\end{smallmatrix}\right).
\end{align*}
Thus $q(G^*)=\lambda_{\max}(E_{G^*})$. Moreover, $E_{G^*}$ is similar to the symmetric matrix
\begin{align*}
    C=
\left(\begin{smallmatrix}
k & \sqrt2 & \sqrt{k-2} & 1 & 0\\
\sqrt2 & 2 & 0 & 0 & 0\\
\sqrt{k-2} & 0 & 3 & 0 & 0\\
1 & 0 & 0 & 2 & \sqrt2\\
0 & 0 & 0 & \sqrt2 & 2
\end{smallmatrix}\right),
\end{align*}
via $C=DE_{G^*}D^{-1}$, where $D=\diag(\sqrt{2}, 1, \sqrt{2(k-2)}, \sqrt{2}, 1)$. Therefore, $q(G^*)=\lambda_{\max}(C)$. We write $(k+1+\tfrac3k)I-C=
\left(\begin{smallmatrix}
1+\frac3k & p^T\\
p & K
\end{smallmatrix}\right)$,
where
\begin{align*}
    p=
\left(\begin{smallmatrix}
-\sqrt2\\
-\sqrt{k-2}\\
-1\\
0
\end{smallmatrix}\right)
\text{ and }
K=
\left(\begin{smallmatrix}
k-1+\frac3k & 0 & 0 & 0\\
0 & k-2+\frac3k & 0 & 0\\
0 & 0 & k-1+\frac3k & -\sqrt2\\
0 & 0 & -\sqrt2 & k-1+\frac3k
\end{smallmatrix}\right).
\end{align*}
The leading principal minors of $K$ are positive for $k\geq 22$, and hence $K$ is positive definite by Lemma \ref{lem9}. Let $S= \left(1+\tfrac3k \right) -p^TK^{-1}p$. Since $K$ is block diagonal, $K^{-1}$ is easy to compute. Writing $a'=k-1+\frac3k$ and $b'=k-2+\frac3k$, we get
\begin{align*}
    K^{-1}= \tfrac{1}{a'b'({a'}^2-2)}
\left(\begin{smallmatrix}
b'({a'}^2-2) & 0 & 0 & 0\\
0 & a'({a'}^2-2) & 0 & 0\\
0 & 0 & {a'}^2b' & \sqrt2{a'}{b'}\\
0 & 0 & \sqrt2{a'}{b'} & {a'}^2{b'}
\end{smallmatrix}\right).
\end{align*}
Using this expression, we obtain $S=1+\frac3k - \frac{2}{a'} - \frac{k-2}{b'} - \frac{a'}{{a'}^2-2}$. After simplification it becomes $S= \frac{10k^6-38k^5+117k^4-216k^3+351k^2-324k+243}{k(k^2-2k+3)(k^2-k+3)(k^4-2k^3+5k^2-6k+9)}$. The numerator can be written as $k^5(10k-38)+k^3(117k-216)+(351k^2-324k+243)$, which is positive for $k\geq 22$. The denominator is also positive in this range. Hence $S>0$. Since $K$ is positive definite and $S>0$, Lemma~\ref{lem10} implies that $\left(k+1+\tfrac3k\right)I-C$ is positive definite. Therefore, by Lemma~\ref{lem8}, $\lambda_{\max}(C)<k+1+\frac{3}{k}$. Since $q(G^*)=\lambda_{\max}(C)$, we get $q(G^*) < k+1+\tfrac3k$, which contradicts \eqref{eq50}.

Next suppose that $w_0=c$. Then the center $c$ has no neighbor in $N$, while every other vertex of $\overline{N}$ has exactly one neighbor in $N$. Thus $G^*$ is precisely the graph $\widetilde{T}$ illustrated in Figure~\ref{fi13}.
\begin{figure}[ht]
    \centering
    \includegraphics[scale=.4]{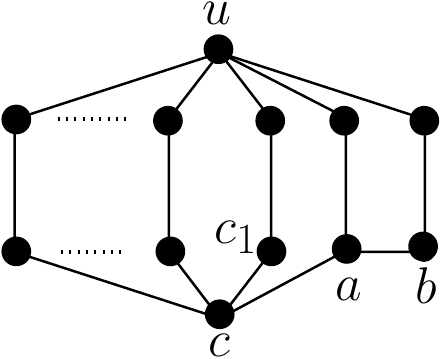}
    \caption{Graph $\widetilde{T}$.}
    \label{fi13}
\end{figure}
Let $u,a,b,c,c_1$ be the vertices as indicated in Figure~\ref{fi13}. If $x_c\geq x_a$, define $G^{*{\prime}}=G^* -\{ab\}+\{cb\}$. Then $G^{*{\prime}} \in \mathscr{G}$, and applying Lemma \ref{lem3}, we get $q(G^{*{\prime}}) > q(G^*)$, a contradiction according to the maximality of $q(G^*)$ in $\mathscr{G}$. If $x_{c} < x_{a}$, we construct $G^{*\prime\prime}=G^* -\{cc_1\}+\{ac_1\}$. Again $G^{*\prime\prime} \in \mathscr{G}$, and Lemma \ref{lem3} gives $q(G^{*\prime\prime}) > q(G^*)$, contradicting the maximality of $q(G^*)$ again. Thus both $w_0=a$ and $w_0=c$ lead to contradictions. Hence Possibility~2 cannot occur.

\textbf{Possibility 3.} $\Delta(H) \le k-2$. As in Possibility~2 of Subcase~1.1, we have $q(H)\le k-1+\frac{2}{k-2}$. With respect to the partition $V(G^*)=\{u\}\cup N\cup \overline{N}$, the signless Laplacian matrix of $G^*$ has the block form $Q(G^*)=
\left(\begin{smallmatrix}
k & \mathbf{1}^T & 0\\
\mathbf{1} & 2I_k & R\\
0 & R^T & Q(H)+E
\end{smallmatrix}\right)$. Here $E$ is the diagonal matrix whose entries are $1$ for the $k$ vertices of $\overline{N}$ adjacent to $N$ and $0$ for $w_0$. Thus $E\leq I_{k+1}$. Also, $R$ corresponds to the adjacency between $N$ and $\overline{N}$; each row of $R$ has exactly one entry equal to $1$, and the $1$'s occur in distinct columns. Hence $RR^T=I_k$ and $\lVert R\rVert=1$. For any vector $z=\left(\begin{smallmatrix}z_1\\z_2\\z_3\end{smallmatrix}\right)$, where $z_1\in\mathbb{R}$, $z_2\in\mathbb{R}^{k}$, $z_3\in\mathbb{R}^{k+1}$, using the preceding estimates, we obtain $z^TQ(G^*)z \le kz_1^2+2\sqrt k\,|z_1|\,\|z_2\|+2\|z_2\|^2 + 2\|z_2\|\,\|z_3\|+(k+\tfrac{2}{k-2})\|z_3\|^2$.
As in Possibility~2 of Subcase~1.1, this gives $q(G^*)\le \lambda_{\max}(L)$, where $L=
\left(\begin{smallmatrix}
k & \sqrt k & 0\\
\sqrt k & 2 & 1\\
0 & 1 & k+\frac{2}{k-2}
\end{smallmatrix}\right)$. Now consider the matrix $(k+1+\frac{3}{k})I -L$. Its leading principal minors are: $D_1=1+\frac{3}{k}$, $D_2=\frac{2k^2+9}{k^2}$ and $D_3= \frac{k^4-3k^3+3k^2-9k-54}{k^4-2k^3}$. For $k\ge 22$, we have $D_1, D_2, D_3 > 0$. Hence, from Lemma \ref{lem9}, $(k+1+\frac{3}{k})I -L$ is positive definite. By Lemma \ref{lem8}, $\lambda_{\max}(L)<k+1+\frac{3}{k}$. Therefore $q(G^*)< k+1+\frac{3}{k}$, which is a contradiction according to \eqref{eq50}. Thus Possibility~3 cannot occur, and consequently Subcase~1.2 cannot occur.

Since neither Subcase~1.1 nor Subcase~1.2 can occur, Case~1 is impossible under the assumption $G^*\not\cong T_k$.
\medskip\\ 
\noindent \textbf{Case 2.} $d(u)=k-1$. Then $|N|=k-1$ and $|\overline{N}|=2k+2-1-(k-1)=k+2$. By Lemma \ref{lem11}\ref{p1}, $k-1\leq e(N,\overline{N})\leq k+2$. Thus, $e(N,\overline{N}) = k-1+l$, where $0\le l \le 3$. We consider two subcases according to the value of $l$.
\medskip\\ 
\noindent \textbf{Subcase 2.1.} $l=3$. Then $e(N,\overline{N})=k+2=|\overline{N}|$. Since each vertex of $\overline{N}$ has at most one neighbor in $N$, equality forces every vertex of $\overline{N}$ to have exactly one neighbor in $N$. For each $v\in N$, let $k_v=e(\{v\},\overline{N})$. Then $k_v\geq 1$ for every $v\in N$, and $\sum_{v\in N}k_v=e(N,\overline{N})=k+2$. Since $|N|=k-1$, we obtain $\sum_{v\in N}(k_v-1)=(k+2)-(k-1)=3$. In particular, $k_v\leq 4$ for every $v\in N$. Let $H=G^*[\overline{N}]$. Since every vertex of $\overline{N}$ has exactly one neighbor in $N$ and $\delta(G^*)\geq 2$, we have $\delta(H)\geq 1$. Recalling the notion of fibre from Subcase~1.1, each vertex of $\overline{N}$ belongs to the fibre determined by its unique neighbor in $N$. As before, a vertex of $\overline{N}$ can neither be adjacent to another vertex in the same fibre nor have more than one neighbor in any other fibre. Since there are $k-1$ fibres in total, $d_H(w)\leq k-2$ for every $w\in V(H)$. Moreover, since $H$ is also $\{C_3,C_4\}$-free, applying Lemma \ref{lem11}\ref{p2} on $H$, we get $\sum_{w'\in N_H(w)}d_H(w')\leq |V(H)|-1=k+1$. If $d_H(w)=1$, then the unique neighbor of $w$ has $H$-degree at most $k-2$, and hence $d_H(w)+\frac{1}{d_H(w)}\sum_{w'\in N_H(w)}d_H(w') \le 1+(k-2)=k-1$. If $2\leq d_H(w)\leq k-2$, then the convexity of 
$f(t)=t+\frac{k+1}{t}$ on $[2,k-2]$ yields $d_H(w)+\frac{1}{d_H(w)}\sum_{w'\in N_H(w)}d_H(w') \leq k-2+\frac{k+1}{k-2} = k-1+\frac{3}{k-2}$. Applying Lemma~\ref{lem2} componentwise to $H$, we get $q(H)\le k-1+\frac{3}{k-2}$. Now, with respect to the partition $V(G^*)=\{u\}\cup N\cup \overline{N}$, the signless Laplacian matrix of $G^*$ has the block form $Q(G^*)= 
\left(\begin{smallmatrix}
k-1 & \mathbf{1}^T & 0\\
\mathbf{1} & D_N & R\\
0 & R^T & Q(H)+I_{k+2}
\end{smallmatrix}\right)$, where $D_N$ is a diagonal matrix with entries $D_N(v,v)=d_{G^*}(v)=1+k_v\le 5$. Also, $R$ corresponds to the adjacency between $N$ and $\overline{N}$. Since every column of $R$ has exactly one entry equal to $1$ and every row has at most $4$ entries equal to $1$, we have $RR^T=\operatorname{diag}(k_v:v\in N)$, and therefore $\lVert R\rVert = \sqrt{\lambda_{\max}(R^TR)} \le 2$. Let $z=
\left(\begin{smallmatrix}
z_1\\
z_2\\
z_3
\end{smallmatrix}\right)$, where $z_1\in\mathbb{R}$, $z_2\in\mathbb{R}^{k-1}$, $z_3\in\mathbb{R}^{k+2}$. Using the estimate for $q(H)$ above, together with $\lVert R\rVert\leq 2$ and $D_N\leq 5I$, we obtain $z^TQ(G^*)z \le (k-1){z_1}^2+2\sqrt{k -1}|z_1| \lVert z_2\rVert +5\lVert z_2\rVert^2 + 4\lVert z_2\rVert \lVert z_3\rVert +\left(k+\tfrac{3}{k-2}\right)\lVert z_3\rVert^2$. By the similar Rayleigh quotient comparison used earlier, from this we get $q(G^*) \le \lambda_{\max}(L)$, where $L=
\left(\begin{smallmatrix}
k-1 & \sqrt{k-1} & 0\\
\sqrt{k-1} & 5 & 2\\
0 & 2 & k+\frac{3}{k-2}
\end{smallmatrix}\right)$. Consider the matrix $(k+1)I -L$. Its leading principal minors are: $D_1=2$, $D_2=k-7$ and $D_3= \frac{(k-17)(k-3)}{k-2}$. For $k\ge 22$, we have $D_1, D_2, D_3 > 0$. Hence, by Lemma \ref{lem9}, $(k+1)I -L$ is positive definite. Therefore, Lemma~\ref{lem8} gives $\lambda_{\max}(L)<k+1<k+1+\frac3k$. Since $q(G^*)\le \lambda_{\max}(L)$, it follows that $q(G^*)< k+1+\frac{3}{k}$, contrary to \eqref{eq50}. Thus Subcase~2.1 cannot occur.
\medskip\\ 
\noindent \textbf{Subcase 2.2.} $0 \le l \le 2$. Since $G^*[N]$ has no edges, we have $\sum_{v\in N}d(v)= d(u)+e(N,\overline{N}) + 2e(N)= (k-1)+(k-1+l) = 2k-2+l$. Thus $M(u) = d(u) + \frac{1}{d(u)}\sum_{v\in N}d(v)= k-1+\frac{2k-2+l}{k-1} =k+1+\frac{l}{k-1} \leq k+1+\frac{2}{k-1}$. For $k\geq 22$, we have $\frac{2}{k-1}<\frac3k$. Therefore, by Lemma \ref{lem2}, $q(G^*) \le M(u) < k+1+\frac3k$, which contradicts \eqref{eq50}. Hence Subcase~2.2 cannot occur. Since neither Subcase~2.1 nor Subcase~2.2 can occur, Case~2 cannot occur.
\medskip\\ 
\noindent \textbf{Case 3.} $3 \le d(u)\le k-2$. From Lemma \ref{lem11}\ref{p3}, we have $q(G^*) \le d(u)+\frac{2k+1}{d(u)}$. The function $f(t)=t+\frac{2k+1}{t}$ is convex for $t>0$, and thus its maximum on $[3,k-2]$ is attained at an endpoint. Now $f(3)=3+\frac{2k+1}{3}=\frac{2k+10}{3}\le k+1+\frac3k$ for $k\geq 22$, and $f(k-2)=k-2+\frac{2k+1}{k-2} = k+\frac{5}{k-2}\le k+1+\frac3k$ for $k\geq 22$. Hence, $q(G^*) \le k+1+\frac3k$. But this contradicts \eqref{eq50}. Therefore, this case cannot arise.
\medskip\\ 
\noindent \textbf{Case 4.} $d(u)= 2$. The discussion is similar to that of Case 4 in the proof of Theorem~\ref{r5}. Let $N=\{a,b\}$. Since $G^*$ is $C_3$-free, $ab\notin E(G^*)$. Define $A=N(a)\setminus\{u\}$, $B=N(b)\setminus\{u\}$ and $N'=V(G^*)\setminus(\{u,a,b\}\cup A\cup B)$. Since $G^*$ is $C_4$-free, no vertex other than $u$ can be adjacent to both $a$ and $b$. Hence $A\cap B=\emptyset$. Let $|A|=r$, $|B|=s$, $|N'|=t$. Then $r+s+t=2k-1$. Also, $d(a)=r+1$ and $d(b)=s+1$. Therefore, $M(u) = 2+\frac{d(a)+d(b)}{2} = 2+\frac{r+s+2}{2} = k+\frac{5}{2}-\frac{t}{2}$. We distinguish the following cases according to the value of $t$.
\medskip\\ 
\noindent \textbf{Subcase 4.1.} $t \ge 3$. Then $M(u)\le k+1<k+1+\frac3k$. Therefore, by Lemma \ref{lem2}, $q(G^*)<k+1+\frac3k$, which contradicts \eqref{eq50}.
\medskip\\ 
\noindent \textbf{Subcase 4.2.} $t = 2$. Then $r+s=2k-3$, and hence one of $r$ and $s$ is at least $k-1$. Without loss of generality, assume that $r\geq k-1$. Since every vertex in $G^*$ has degree at least $2$, we get $\sum_{v\in N(a)}d(v) \ge 2(r+1)$. Thus $M(a) = d(a) +\frac{1}{d(a)}\sum_{v\in N(a)}d(v) \ge r+1+\frac{2(r+1)}{r+1} = r+3 \ge k+2$. On the other hand, $M(u) = k+\frac{5}{2}-\frac{t}{2}$ gives $M(u)=k+\frac32$. Hence $M(a)>M(u)$, a contradiction. Therefore this subcase cannot occur.
\medskip\\ 
\noindent \textbf{Subcase 4.3.} $t = 1$. Then $r+s=2k-2$. First, suppose that one of $r$ and $s$ is at least $k$. Without loss of generality, assume that $r\geq k$. As in Subcase 4.2, we obtain $M(a)\ge r+3 \ge k+3$. But $M(u) = k+\frac{5}{2}-\frac{t}{2}$ gives $M(u)=k+2$. Thus $M(a) > M(u)$, a contradiction again. Hence neither $r$ nor $s$ can be at least $k$, and therefore $r=s=k-1$. In this scenario, $G^*$ is precisely $\theta(2,3^{k-2},4)$. In Subcase 1.2, Possibility 2, it is already established that $q\bigl(\theta(2,3^{k-2},4)\bigr)<k+1+\frac{3}{k}$. Therefore, $q(G^*) < k+1+\frac3k$ here, which contradicts \eqref{eq50}. Hence Subcase 4.3 can not occur.
\medskip\\ 
\noindent \textbf{Subcase 4.4.} $t = 0$. Then $r+s=2k-1$, so one of $r,s$ is at least $k$. Without loss of generality, assume that $r\geq k$. Using the same argument as in Subcase 4.3, we obtain $M(a)\geq r+3\geq k+3$. However, $M(u) = k+\frac{5}{2}-\frac{t}{2}$ gives $M(u)=k+\frac{5}{2}$. Thus $M(a)>M(u)$, a contradiction. Therefore, Subcase~4.4 cannot occur, and so neither can Case~4.

We have now shown that, under the assumption $G^*\not\cong T_k$, every possible value of $d(u)$ leads to a contradiction. Therefore, $G^*\cong T_k=T_{\frac{n-2}{2}}$. Moreover, substituting $k=\frac{n-2}{2}$ into the expression for $j(x)$, we obtain that the maximum signless Laplacian spectral radius is the largest root of $x^5-\frac{1}{2}(n+18)x^4 +\frac{3}{2}(3n+16)x^3 -\frac{1}{2}(27n+20)x^2 + 4(4n-9)x -6n+24=0$. This proves Theorem~\ref{r6}. \qedhere
\end{proof}


\section{Fixed-size extremal characterization for $\{\theta(1,2,2)$, $\theta(1,2,3)\}$-free graphs} \label{sec4}

In this section, we determine the graphs attaining the maximum signless Laplacian spectral radius among all $\{\theta(1,2,2),\theta(1,2,3)\}$-free graphs of size $m$ with no pendant vertices, separately for $m\equiv 1,2\pmod 3$. We continue to use the notations introduced in Notation~\ref{n1}. Recall that $\mathcal{G}$ denotes the class of $\{\theta(1,2,2),\theta(1,2,3)\}$-free graphs of size $m$ with no pendant vertices, $G^*$ denotes a graph attaining the maximum signless Laplacian spectral radius in $\mathcal{G}$, and $u$ is a vertex of $G^*$ satisfying $M(u)=\max\{M(v):v\in V(G^*)\}$. We also write $e'$ for $e(N,\overline{N})$ in $G^*$. We first establish two lemmas that will be used in the sequel.

\begin{lemma}\label{lem12}
For the graph $G^*\in\mathcal{G}$, the following hold:
\begin{enumerate}[topsep=4pt,itemsep=0pt,label=\upshape(\roman*),ref=\upshape(\roman*)]
    \item\label{p4} $G^*[N]=tK_2\cup sK_1$, where $0\le t\le \frac{d(u)}{2}$, $0\le s\le d(u)$ and $d(u)=2t+s$;
    \item\label{p5} $\frac{3d(u)}{2}-\frac{s}{2}+e'+e(\overline{N})=m$, where $s$ is as in \ref{p4}; and
    \item\label{p6} $2\le d(u)\le \frac{2m}{3}$.
\end{enumerate}
\end{lemma}

\begin{proof}
\begin{enumerate}[topsep=4pt,itemsep=0pt,label=\upshape(\roman*),ref=\upshape(\roman*)]
    \item Since $G^*$ is $\{\theta(1,2,2),\theta(1,2,3)\}$-free, every vertex of $G^*[N]$ has degree at most $1$ in $G^*[N]$; otherwise, together with $u$, a copy of $\theta(1,2,2)$ would arise. Hence the result follows.
    \item By \ref{p4}, $G^*[N]=tK_2\cup sK_1$ implies $e(N)=t$, and $d(u)=2t+s$ implies $t=\frac{d(u)-s}{2}$. Using these, we obtain $m=d(u)+e(N)+e'+e(\overline{N})
    =d(u)+t+e'+e(\overline{N})
    =\frac{3d(u)}{2}-\frac{s}{2}+e'+e(\overline{N})$.
    \item Since $G^*$ is connected and has no pendant vertices, we have $d(u)\geq 2$. Moreover, each isolated vertex of $G^*[N]$ must have a neighbor in $\overline{N}$, and hence $e'\geq s$. Therefore, by \ref{p5}, $m\geq \frac{3d(u)}{2}-\frac{s}{2}+s
    =\frac{3d(u)}{2}+\frac{s}{2}
    \geq \frac{3d(u)}{2}$. Thus $d(u)\leq \frac{2m}{3}$, and the proof is complete.\qedhere
\end{enumerate}
\end{proof}

\begin{lemma}\label{lem6}
    Let $m=3k+1$ with $k\ge4$. If
$G^*[N]=tK_2\cup sK_1$,
$e'=s$, and
$d(u)=\frac{2m-5}{3}$,
then $q(G^*)\leq\frac{2m+1}{3}$.
\end{lemma}
\begin{proof}
    Applying Lemma \ref{lem2}, we obtain $q(G^*) \leq d(u) + \frac{d(u) + 2e(N)+ e'}{d(u)}=d(u) + \frac{d(u)+2t+s}{d(u)}=d(u)+\frac{d(u)+d(u)}{d(u)}=d(u)+2=\frac{2m+1}{3}$.\qedhere
\end{proof}

\begin{theorem}\label{r1}
Let $G$ be a $\{\theta (1, 2, 2), \theta (1, 2, 3)\}$-free graph of size $m = 3k+1$, $k \geq 4$, with no pendant vertices. Then $q(G) \leq q'(m)$, where $q'(m)$ is the largest root of $x^4-\frac{1}{3}(2m+19)x^3 + (4m+10)x^2 -\frac{4}{3}(5m-2)x+ \frac{8}{3}(m-4)=0$. Moreover, equality holds if and only if $G \cong H_{\frac{2m-5}{3}}^{(4)}$.
\end{theorem}

\begin{proof}
Let $G^*$ be a graph with the maximum signless Laplacian spectral radius in the class $\mathcal{G}$ with $m=3k+1$, $k\geq 4$. The number $\frac{2m-5}{3}$ is an integer, and hence the graph $H_{\frac{2m-5}{3}}^{(4)}$ is well defined. $H_{\frac{2m-5}{3}}^{(4)}$ belongs to $\mathcal{G}$ and has maximum degree $\frac{2m-2}{3}$. The signless Laplacian matrix
$Q\bigl(H_{\frac{2m-5}{3}}^{(4)}\bigr)$ admits an equitable partition whose partition classes consist of the common vertex of $F_{\frac{2m-5}{3}}$ and $C_4$, the remaining vertices of $F_{\frac{2m-5}{3}}$ adjacent to it, the two remaining vertices of $C_4$ adjacent to the common vertex, and the remaining vertex of $C_4$. The corresponding quotient matrix is
\begin{align*}
E_{H_{\frac{2m-5}{3}}^{(4)}}=
\left(\begin{smallmatrix}
\frac{2m-2}{3} & \frac{2m-8}{3} & 2 & 0\\
1 & 3 & 0 & 0\\
1 & 0 & 2 & 1\\
0 & 0 & 2 & 2
\end{smallmatrix}\right).
\end{align*}
Its characteristic polynomial is $h(x)=x^4-\frac{1}{3}(2m+19)x^3+(4m+10)x^2-\frac{4}{3}(5m-2)x+\frac{8}{3}(m-4)$. If $\lambda_h$ is the largest root of $h(x)=0$, then Lemma~\ref{lem4} gives $q\bigl(H_{\frac{2m-5}{3}}^{(4)}\bigr)=\lambda_h$. Moreover, since $H_{\frac{2m-5}{3}}^{(4)}$ is not a star, Lemma~\ref{lem1} implies $q\bigl(H_{\frac{2m-5}{3}}^{(4)}\bigr)>
\frac{2m-2}{3}+1
=
\frac{2m+1}{3}$. Now, if possible, let $G^*\not\cong H_{\frac{2m-5}{3}}^{(4)}$. Then
\begin{align}\label{eq3}
q(G^*)
\geq
q\bigl(H_{\frac{2m-5}{3}}^{(4)}\bigr)
>
\frac{2m+1}{3}.
\end{align}
From Lemma \ref{lem12}\ref{p4}, $G^*[N]=tK_2\cup sK_1$, where $0\le t\le \frac{d(u)}{2}$, $0\le s\le d(u)$ and $d(u)=2t+s$.  Since $m=3k+1$, the number $\frac{2m}{3}$ is not an integer, and $\left\lfloor\frac{2m}{3}\right\rfloor
=
\frac{2m-2}{3}$.
Therefore, by Lemma~\ref{lem12}\ref{p6}, $2 \leq d(u)\leq \frac{2m-2}{3}$.
We distinguish the following cases according to the possible values of $d(u)$.
\medskip\\ 
\noindent \textbf{Case 1.} $d(u)=\frac{2m-2}{3}$. Applying Lemma~\ref{lem7} to $G^*$, we get $2m \geq 3\times \frac{2m-2}{3}+2|\overline{N}|$, and hence $|\overline{N}|\leq 1$. Thus $|\overline{N}|=0$ or $1$.
\medskip\\
\noindent\textbf{Subcase 1.1.} $|\overline{N}|=0$. Then $e'=0$ and $e(\overline{N})=0$. Hence, by Lemma~\ref{lem12}\ref{p5}, $m=\frac{3d(u)}{2}-\frac{s}{2}$. Since $d(u)=\frac{2m-2}{3}$ and $s\geq 0$, we obtain $m\leq \frac{3d(u)}{2}=m-1$, a contradiction. Therefore Subcase~1.1 cannot occur.
\medskip\\
\noindent\textbf{Subcase 1.2.} $|\overline{N}|=1$. Let $\overline{N}=\{w\}$. We first observe that $w$ can be adjacent only to isolated vertices of $G^*[N]$. Indeed, if $w$ was adjacent to a vertex of $N$ incident with an edge in $G^*[N]$, then, since $d(w)\geq 2$, the vertex $w$ would have another neighbor in $N$. If this second neighbor is the other endpoint of the same edge in $G^*[N]$, then a copy of $\theta(1,2,2)$ is formed; otherwise, a copy of $\theta(1,2,3)$ is formed. Hence $w$ is adjacent only to isolated vertices of $G^*[N]$. Since there are exactly $s$ isolated vertices in $G^*[N]$, we have $e'\leq s$. But from the proof of Lemma~\ref{lem12}\ref{p6}, we already know that $e'\geq s$. Hence $e'=s$. Moreover, $e(\overline{N})=0$ and $d(u)=\frac{2m-2}{3}$. Using these, from Lemma~\ref{lem12}\ref{p5}, we obtain $s=2$, and it follows that $G^*\cong H_{\frac{2m-5}{3}}^{(4)}$,
contradicting the assumption $G^*\not\cong H_{\frac{2m-5}{3}}^{(4)}$. Therefore Subcase~1.2 cannot occur, and hence neither can Case~1.
\medskip\\ 
\noindent \textbf{Case 2.} $d(u)=\frac{2m-2}{3}-1=\frac{2m-5}{3}$. First suppose that $e'=s$. Then Lemma~\ref{lem6} gives $q(G^*)\leq \frac{2m+1}{3}$,
contrary to \eqref{eq3}. Hence $e'>s$. From Lemma~\ref{lem7}, $2m\geq 3\times \frac{2m-5}{3}+2|\overline{N}|$, and thus $|\overline{N}|\leq \frac52$. Hence $|\overline{N}|=0,1$ or $2$. If $|\overline{N}|=0$, then $e'=0$. Since $s\geq0$, we cannot have $e'>s$. If $|\overline{N}|=1$, then, as in Subcase~1.2, we necessarily have $e'=s$, again a contradiction. Thus the only remaining possibility is $|\overline{N}|=2$. Let $\overline{N}=\{w_1,w_2\}$. We now distinguish the following two cases according to the possible structures of $G^*[\overline{N}]$.
\medskip\\
\noindent\textbf{Subcase 2.1.} $G^*[\overline{N}]=2K_1$. Since $e(\overline{N})=0$, Lemma~\ref{lem12}\ref{p5} gives $m=\frac{3d(u)}{2}-\frac{s}{2}+e'$. Substituting $d(u)=\frac{2m-5}{3}$, we obtain $e'=\frac{s+5}{2}$. Since $m=3k+1$, we have $d(u)=\frac{2m-5}{3}=2k-1$,
which is odd. Therefore $s$ is also odd, by $d(u)=2t+s$. This, together with $e'>s$ and $e'=\frac{s+5}{2}$, implies that $(s,e')=(1,3)$ or $(3,4)$. The pair $(1,3)$ is impossible. Indeed, since $G^*[\overline{N}]=2K_1$ and $\delta(G^*)\geq 2$, each of $w_1$ and $w_2$ must have at least two neighbors in $N$, implying $e'\ge4$, contradicting $e'=3$. Therefore, it remains to consider $(s,e')=(3,4)$. In this case, each of $w_1$ and $w_2$ has exactly two neighbors in $N$. Moreover, no vertex of $\overline{N}$ can be adjacent to a vertex incident with an edge in $G^*[N]$; otherwise, together with $u$ and another neighbor in $N$, a copy of $\theta(1,2,2)$ or $\theta(1,2,3)$ would arise. Hence every edge between $N$ and $\overline{N}$ is incident with an isolated vertex of $G^*[N]$. Since $s=3$ and $e'=4$, this determines, up to isomorphism, the unique graph $G_1$ shown in Figure~\ref{fi2}(a).
\begin{figure}[ht]
    \centering
    \includegraphics[scale=.4]{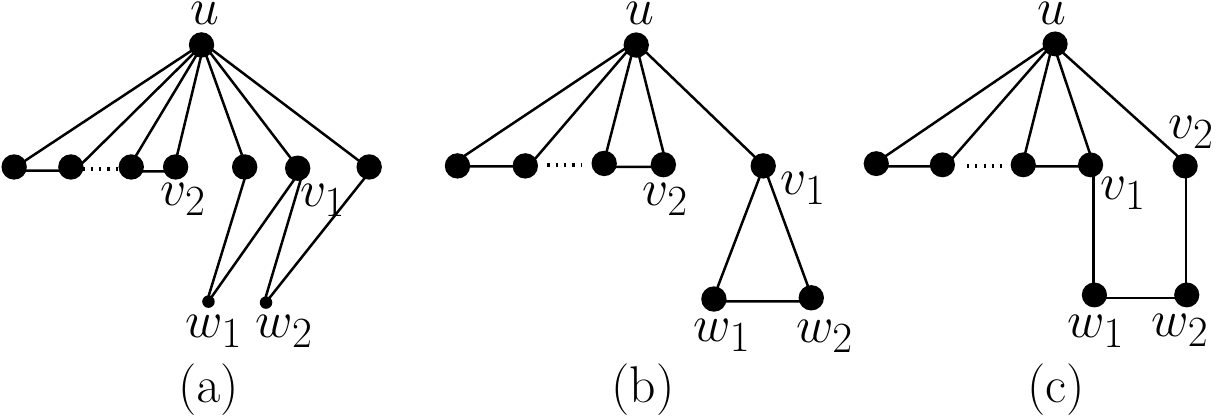}
    \caption{The graphs (a) $G_1$, (b) $G_2$ and (c) $G_3$.}
    \label{fi2}
\end{figure}
Suppose $G^*\cong G_1$, and let $w_1,w_2,u,v_1,v_2$ be the vertices as indicated in Figure~\ref{fi2}(a). If $x_u\geq x_{v_1}$, define $G^{*{\prime}}=G^* -\{v_1w_1\}+\{u w_1\}$. Then $G^{*{\prime}} \in \mathcal{G}$, and by Lemma~\ref{lem3}, $q(G^{*{\prime}}) > q(G^*)$, contradicting the maximality of $q(G^*)$ in $\mathcal{G}$. If $x_{u} < x_{v_1}$, define $G^{*\prime\prime}=G^* -\{uv_2\}+\{v_1v_2\}$. Again, $G^{*\prime\prime}\in\mathcal{G}$, and Lemma~\ref{lem3} gives $q(G^{*\prime\prime}) > q(G^*)$, yielding the same contradiction. Therefore Subcase~2.1 cannot occur.
\medskip\\
\noindent \textbf{Subcase 2.2.} $G^*[\overline{N}]=K_2$. Since $e(\overline{N})=1$ and $d(u)=\frac{2m-5}{3}$, by Lemma~\ref{lem12}\ref{p5}, we obtain $e'=\frac{s+3}{2}$. Since $d(u)=\frac{2m-5}{3}=2k-1$ is odd, $s$ is also odd by $d(u)=2t+s$. These, together with $e'>s$, implies that $(s,e')=(1,2)$. Now the unique isolated vertex of $G^*[N]$ must be adjacent to at least one vertex of $\overline{N}$. Since $G^*[\overline{N}]=K_2$ and $e'=2$, each of $w_1$ and $w_2$ has exactly one neighbor in $N$. Hence, up to isomorphism, there are only two possible configurations: either both vertices of $\overline{N}$ are adjacent to the isolated vertex of $G^*[N]$, or one of them is adjacent to that isolated vertex while the other is adjacent to an endpoint of an edge in $G^*[N]$. These correspond precisely to the graphs $G_2$ and $G_3$ shown in Figure~\ref{fi2}(b) and Figure~\ref{fi2}(c), respectively. Suppose $G^*\cong G_i$ for some $i\in\{2,3\}$, and let the vertices $w_1,w_2,u,v_1,v_2$ be as indicated in each of Figure~\ref{fi2}(b) and \ref{fi2}(c). If $x_u\geq x_{v_1}$, define $G^{*{\prime}}=G^* -\{v_1w_1\}+\{u w_1\}$. Then $G^{*{\prime}} \in \mathcal{G}$. Applying Lemma \ref{lem3}, we get $q(G^{*{\prime}}) > q(G^*)$, a contradiction. If $x_{u} < x_{v_1}$, we define $G^{*\prime\prime}=G^* -\{u v_2\}+\{v_1v_2\}$. Again $G^{*\prime\prime} \in \mathcal{G}$, and Lemma \ref{lem3} gives $q(G^{*\prime\prime}) > q(G^*)$, a contradiction. Therefore Subcase~2.2 cannot occur. Consequently, Case~2 cannot occur.
\medskip\\ 
\noindent \textbf{Case 3.}  $2 \leq d(u) \leq \frac{2m-2}{3} -2= \frac{2m-8}{3}$. By Lemma~\ref{lem2}, we have $q(G^*)
\leq
d(u)+\frac{d(u)+2e(N)+e'}{d(u)}
=
d(u)+\frac{d(u)+2t+e'}{d(u)}$. Since $e'\leq m-d(u)-t$ and $t\leq \frac{d(u)}{2}$, we obtain $q(G^*) \leq d(u) + \frac{m}{d(u)} + \frac{1}{2}$. For $2\leq x\leq \frac{2m-8}{3}$, define $F(x)=\frac{2m+1}{3}-\left(x+\frac{m}{x}+\frac{1}{2}\right)$. Then $F''(x)=-\frac{2m}{x^3}<0$, so $F(x)$ is concave on this interval. Hence $F(x)\geq \min\left\{ F(2), F\left(\frac{2m-8}{3}\right)\right\}$. Now $F(2) = \frac{2m+1}{3}-2-\frac{m}{2}-\frac{1}{2}=\frac{m-13}{6}\geq 0$ for $m\geq 13$. Also, $F(\frac{2m-8}{3}) = \frac{2m+1}{3}-\frac{2m-8}{3}-\frac{3m}{2m-8}-\frac{1}{2}=\frac{m-10}{m-4}> 0$ for $m\geq 13$. Since $m=3k+1$ with $k\geq 4$, we have $m\geq 13$. Therefore, $F(x) \geq 0$ for $2\leq x\leq \frac{2m-8}{3}$. Substituting $x=d(u)$ into it, gives $d(u)+\frac{m}{d(u)}+\frac{1}{2} \leq \frac{2m+1}{3}$. Combining this with $q(G^*) \leq d(u) + \frac{m}{d(u)} + \frac{1}{2}$ yields $q(G^*) \leq \frac{2m+1}{3}$, a contradiction to \eqref{eq3}. Therefore Case~3 cannot occur.

We have shown that, under the assumption $G^*\not\cong H_{\frac{2m-5}{3}}^{(4)}$, none of the Cases~1--3 can occur. Therefore $G^*\cong H_{\frac{2m-5}{3}}^{(4)}$. Furthermore, from the characteristic polynomial of the quotient matrix of
$Q(H_{\frac{2m-5}{3}}^{(4)})$, the maximum signless Laplacian spectral radius is the largest root of $x^4-\frac13(2m+19)x^3
+(4m+10)x^2
-\frac43(5m-2)x
+\frac83(m-4)=0$. This proves Theorem~\ref{r1}.\qedhere
\end{proof}


We next establish two lemmas that will be needed in the sequel.

\begin{lemma}\label{lem13}
Let $m=3k+2$, where $k\geq 6$, and let $G_4$ be the graph obtained by identifying a terminal vertex of $\theta(2^4)$ with the central vertex of the friendship graph $F_{\frac{2m-13}{3}}$. Then $q(G_4)<q\bigl(H_{\frac{2m-7}{3}}^{(5)}\bigr)$.
\end{lemma}

\begin{proof}
Both the graphs $G_4$ and $H_{\frac{2m-7}{3}}^{(5)}$ have $m$ edges. The matrix $Q(G_4)$ admits an equitable partition whose classes consist of the common vertex of $F_{\frac{2m-13}{3}}$ and $\theta(2^4)$, the remaining vertices of the friendship graph, the four internal vertices of $\theta(2^4)$, and the other terminal vertex of $\theta(2^4)$. The corresponding quotient matrix is
\[
E_{G_4}=
\left(\begin{smallmatrix}
\frac{2m-4}{3} & \frac{2m-16}{3} & 4 & 0\\
1 & 3 & 0 & 0\\
1 & 0 & 2 & 1\\
0 & 0 & 4 & 4
\end{smallmatrix}\right),
\]
whose characteristic polynomial is $g(x)=x^4-\frac{1}{3}(2m+23)x^3
+\frac{2}{3}(8m+17)x^2
-\frac{8}{3}(4m-5)x
+\frac{16}{3}(m-8)$. Similarly, $Q\bigl(H_{\frac{2m-7}{3}}^{(5)}\bigr)$ admits an equitable partition whose classes consist of the common vertex of $F_{\frac{2m-7}{3}}$ and $C_5$, the remaining vertices of the friendship graph, the two neighbors of the common vertex on $C_5$, and the remaining two vertices of $C_5$. The corresponding quotient matrix is
\[
E_{H_{\frac{2m-7}{3}}^{(5)}}=
\left(\begin{smallmatrix}
\frac{2m-4}{3} & \frac{2m-10}{3} & 2 & 0\\
1 & 3 & 0 & 0\\
1 & 0 & 2 & 1\\
0 & 0 & 1 & 3
\end{smallmatrix}\right),
\]
whose characteristic polynomial is $h(x)=x^4-\frac{2}{3}(m+10)x^3
+\frac{2}{3}(7m+16)x^2
-(10m-7)x
+\frac{4}{3}(5m-16)$. Let $\lambda_g$ and $\lambda_h$ denote the largest roots of $g(x)=0$
and $h(x)=0$, respectively. Then, by Lemma~\ref{lem4}, $q(G_4)=\lambda_g$ and $q\bigl(H_{\frac{2m-7}{3}}^{(5)}\bigr)=\lambda_h$. Both $G_4$ and $H_{\frac{2m-7}{3}}^{(5)}$ have maximum degree $\frac{2m-4}{3}$ and neither of them is a tree. Hence, by Lemma~\ref{lem1}, $q(G_4)>\frac{2m-1}{3}$ and $q\bigl(H_{\frac{2m-7}{3}}^{(5)}\bigr)>\frac{2m-1}{3}$. Moreover, Lemma~\ref{lem2} gives $q(G_4)\leq \frac{2m-4}{3}+\frac{\frac{2m-4}{3}\times 2}{\frac{2m-4}{3}} =\frac{2m+2}{3}$. The same bound is obtained for $q\bigl(H_{\frac{2m-7}{3}}^{(5)}\bigr)$ as well. Therefore, $\frac{2m-1}{3}<\lambda_g,\lambda_h\leq\frac{2m+2}{3}$. We prove the lemma by establishing the following three claims.
\smallskip\\
\noindent\textbf{Claim A.} Both $g(x)$ and $h(x)$ are strictly increasing functions for
$x\geq\frac{2m-1}{3}$.
\smallskip\\
\noindent\textbf{Proof of Claim A.} We obtain $g'(x)=4x^3-(2m+23)x^2
+\frac{4}{3}(8m+17)x-\frac{8}{3}(4m-5)$,
$g''(x)=12x^2-2(2m+23)x+\frac{4}{3}(8m+17)$,
$g^{(3)}(x)=24x-2(2m+23)$,
and $g^{(4)}(x)=24$.
Since $m=3k+2$ with $k\geq 6$, we have $m\geq 20$. As $g^{(4)}(x)>0$, the function $g^{(3)}(x)$ is increasing. Hence, for $x\geq \frac{2m-1}{3}$, $g^{(3)}(x)\geq g^{(3)}\left(\frac{2m-1}{3}\right)= 12m-54>0$. Thus $g''(x)$ is increasing on this interval, and $g''(x)\geq g''\left(\frac{2m-1}{3}\right) = \frac{8}{3}m^2-24m+\frac{118}{3}>0$ for $m\geq 20$. Consequently, $g'(x)$ is increasing, and $g'(x)\geq g'\left(\frac{2m-1}{3}\right) = \frac{8}{27}m^3-4m^2+\frac{106}{9}m+\frac{83}{27}>0$ for $m\geq 20$. Therefore, $g(x)$ is strictly increasing for $x\geq \frac{2m-1}{3}$.

The proof for $h(x)$ is analogous: successive derivatives of $h$ give the same positivity pattern on $x\geq \frac{2m-1}{3}$ for $m\geq 20$. Hence $h(x)$ is also strictly increasing on this interval.
\smallskip\\
\noindent\textbf{Claim B.}
At $x=\frac{2m-1}{3}$, $h(x)<g(x)<0$,
whereas at $x=\frac{2m+2}{3}$, $h(x)>g(x)\geq0$.
\smallskip\\
\noindent\textbf{Proof of Claim B.}
Since $\frac{2m-1}{3}<\lambda_g,\lambda_h\leq\frac{2m+2}{3}$, Claim A implies that both $g(x)$ and $h(x)$ have unique zeros in $\left[\frac{2m-1}{3},\frac{2m+2}{3}\right]$, and $g\left(\frac{2m-1}{3}\right)<0$, $h\left(\frac{2m-1}{3}\right)<0$, $g\left(\frac{2m+2}{3}\right)\geq 0$, $h\left(\frac{2m+2}{3}\right)\geq 0$. Now we define $D(x)=g(x)-h(x)$. Then for $m\geq 20$, $D\left(\frac{2m-1}{3}\right)
=\frac{8m-70}{3}>0$. Also $D\left(\frac{2m+2}{3}\right)
=\frac{-4m^2+22m-154}{9}<0$, since the numerator has negative leading coefficient and negative discriminant. Hence the claim follows.
\smallskip\\
\noindent\textbf{Claim C.}
If $c\in\left[\frac{2m-1}{3},\frac{2m+2}{3}\right]$ and $g(c)=h(c)$,
then $g(c)=h(c)>0$.
\smallskip\\
\noindent\textbf{Proof of Claim C.}
Since $g(c)=h(c)$, we have $D(c)=0$. Further, this gives
\begin{align}\label{eq18}
    c^3=\frac{2}{3}(m+1)c^2 -\frac{1}{3}(2m-19)c -\frac{4}{3}(m+16)
\end{align}
Multiplying both sides of this relation by $c$, we get $c^4=\frac{2}{3}(m+1)c^3 -\frac{1}{3}(2m-19)c^2 -\frac{4}{3}(m+16)c$. Applying \eqref{eq18} again to the right-hand side gives
\begin{align}\label{eq19}
    c^4=\frac{1}{9}(4m^2+2m+61)c^2 -\frac{2}{9}(2m^2-11m+77)c -\frac{8}{9}(m^2+17m+16).
\end{align}
Using both \eqref{eq18} and \eqref{eq19} into the expression for $h(c)$, yields
\begin{align}\label{eq20}
    h(c)=13c^2-\frac{1}{3}(22m+157)c +\frac{4}{3}(11m+80).
\end{align}
Since $c\in \left[\frac{2m-1}{3}, \frac{2m+2}{3}\right]$, we can write $c=\frac{2m-1}{3}+\left(\frac{2m+2}{3}-\frac{2m-1}{3}\right)t$, thus, $c=\frac{2m-1}{3} +t$, where $0 \le t \le 1$.
Substituting this into \eqref{eq20}, we obtain $h(c) =\frac{8}{9}m^2 +\frac{2}{9}(45t-106)m+\frac{1}{9}(117t^2 -549t+1130)=: F(m)$. For fixed $t\in[0,1]$, the polynomial $F(m)$ is quadratic in $m$ with positive leading coefficient, thus representing an upward opening parabola, with its vertex at $(m_v, F(m_v))$, where $m_v=-\frac{\frac{2}{9}(45t-106)}{\frac{16}{9}}=\frac{212-90t}{16}<20$. Moreover, $F(20)=13t^2+139t+10>0$ for $0\leq t\leq 1$. Hence $F(m)>0$ for all $m\geq 20$. Therefore, $h(c)>0$. Since $g(c)=h(c)$, we get $g(c)=h(c)>0$. This proves Claim C.

By Claims A and B, the graphs of $g$ and $h$ intersect the $x$-axis exactly once in the interval $\left[\frac{2m-1}{3},\frac{2m+2}{3}\right]$, namely at $\lambda_g$ and $\lambda_h$, respectively. Moreover, Claim~B shows that the graphs of $g$ and $h$ intersect each other at least once in this interval. Claim C shows that every such intersection point lies above the $x$-axis. Since $h\left(\frac{2m-1}{3}\right) < g\left(\frac{2m-1}{3}\right)<0$, the zero of $g$ occurs before the zero of $h$. Therefore $\lambda_g<\lambda_h$, and hence $q(G_4) < q\bigl(H_{\frac{2m-7}{3}}^{(5)}\bigr)$. \qedhere
\end{proof}

\begin{lemma}\label{lem5}
    Let $m=3k+2$ with $k\ge6$. If
$G^*[N]=tK_2\cup sK_1$,
$e'=s$, and
$d(u)=\frac{2m-7}{3}$,
then $q(G^*)\le\frac{2m-1}{3}$.
\end{lemma}

\begin{proof}
    Applying Lemma \ref{lem2}, we obtain $q(G^*) \leq d(u) + \frac{d(u) + 2e(N)+ e'}{d(u)}=d(u) + \frac{d(u)+2t+s}{d(u)}=d(u)+\frac{d(u)+d(u)}{d(u)}=d(u)+2=\frac{2m-1}{3}$.\qedhere
\end{proof}

\begin{theorem}\label{r2}
Let $G$ be a $\{\theta (1, 2, 2), \theta (1, 2, 3)\}$-free graph of size $m = 3k+2$, $k \geq 6$, with no pendant vertices. Then $q(G) \leq q'(m)$, where $q'(m)$ is the largest root of $x^4-\frac{2}{3}(m+10)x^3 + \frac{2}{3}(7m+16)x^2 -(10m-7)x+ \frac{4}{3}(5m-16)=0$. Moreover, equality holds if and only if $G \cong H_{\frac{2m-7}{3}}^{(5)}$.
\end{theorem}

\begin{proof}
Let $G^*$ be a graph with the maximum signless Laplacian spectral radius in the class $\mathcal{G}$ with $m=3k+2$, $k\geq 6$. The graph $H_{\frac{2m-7}{3}}^{(5)}$ belongs to $\mathcal{G}$. In the proof of Lemma~\ref{lem13}, it was shown that $q\bigl(H_{\frac{2m-7}{3}}^{(5)}\bigr)$ is the largest root of $h(x)=0$, where $h(x)=x^4-\frac{2}{3}(m+10)x^3
+\frac{2}{3}(7m+16)x^2
-(10m-7)x
+\frac{4}{3}(5m-16)$, and that $q(H_{\frac{2m-7}{3}}^{(5)}) > \frac{2m-1}{3}$.
Suppose, if possible, that $G^*\not\cong H_{\frac{2m-7}{3}}^{(5)}$. Then
\begin{align}\label{eq11}
    q(G^*) \geq q(H_{\frac{2m-7}{3}}^{(5)}) > \frac{2m-1}{3}.
\end{align} 
By Lemma \ref{lem12}\ref{p4}, $G^*[N]=tK_2\cup sK_1$, where $0\le t\le \frac{d(u)}{2}$, $0\le s\le d(u)$ and $d(u)=2t+s$.  Since $m=3k+2$, the number $\frac{2m}{3}$ is not an integer, and $\left\lfloor\frac{2m}{3}\right\rfloor
=
\frac{2m-1}{3}$.
Therefore, by Lemma~\ref{lem12}\ref{p6}, $2 \leq d(u)\leq \frac{2m-1}{3}$.
We distinguish the following cases according to the possible values of $d(u)$.
\medskip\\ 
\noindent \textbf{Case 1.} $d(u)=\frac{2m-1}{3}$. By Lemma \ref{lem7}, $2m \geq 3\times \frac{2m-1}{3} +2|\overline{N}|$, which gives $|\overline{N}| \leq \frac{1}{2}$. Hence $|\overline{N}|=0$, and so $e'=0$ and $e(\overline{N})=0$. By Lemma~\ref{lem12}\ref{p5}, we obtain $s=-1$, a contradiction. Therefore, Case 1 cannot occur.
\medskip\\ 
\noindent \textbf{Case 2.} $d(u)=\frac{2m-1}{3}-1=\frac{2m-4}{3}$. By Lemma~\ref{lem7}, $2m\geq 3\times\frac{2m-4}{3}+2|\overline{N}|$, and hence $|\overline{N}|\leq 2$. Therefore
$|\overline{N}|=0,1$ or $2$.
\medskip\\
\noindent\textbf{Subcase 2.1.} $|\overline{N}|=0$. Then $e'=0$ and $e(\overline{N})=0$. Substituting these together with
$d(u)=\frac{2m-4}{3}$ into Lemma~\ref{lem12}\ref{p5}, we obtain
$s=-4$, a contradiction.
\medskip\\
\noindent\textbf{Subcase 2.2.} $|\overline{N}|=1$. Let $\overline{N}=\{w\}$. As in Subcase~1.2 of the proof of Theorem~\ref{r1}, the vertex $w$ can be adjacent only to isolated vertices of $G^*[N]$. Hence $e'\leq s$. But from the proof of Lemma~\ref{lem12}\ref{p6}, we already have $e'\geq s$, therefore $e'=s$. Also, $|\overline{N}|=1$ implies $e(\overline{N})=0$. Substituting these together with
$d(u)=\frac{2m-4}{3}$ into Lemma~\ref{lem12}\ref{p5}, we obtain $s=4$, and hence $e'=4$. Thus the unique vertex of $\overline{N}$ is adjacent precisely to the four isolated vertices of $G^*[N]$, while all remaining vertices of $N$ are paired by the edges of $G^*[N]$. Therefore, up to isomorphism, the only possible configuration in $\mathcal{G}$ is the graph $G_4$ shown in Figure~\ref{fi3}.
\begin{figure}[ht]
    \centering
    \includegraphics[scale=.4]{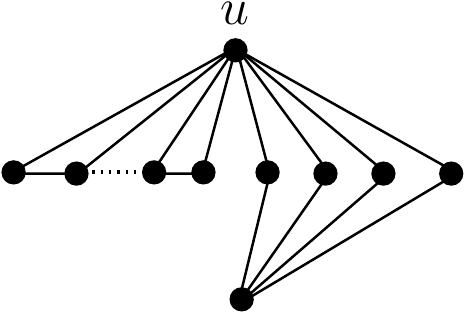}
    \caption{Graph $G_4$.}
    \label{fi3}
\end{figure}
If possible, let $G^*\cong G_4$. Then Lemma~\ref{lem13} gives $q(G^*)<q\bigl(H_{\frac{2m-7}{3}}^{(5)}\bigr)$, contradicting the maximality of $q(G^*)$ in $\mathcal{G}$. Hence Subcase~2.2 cannot occur.
\medskip\\ 
\noindent \textbf{Subcase 2.3.} $|\overline{N}|=2$. Let $\overline{N}=\{w_1,w_2\}$. There are two possibilities for the structure of $G^*[\overline{N}]$.

\textbf{Possibility 1.} $G^*[\overline{N}]=2K_1$. Then $e(\overline{N})=0$. From Lemma~\ref{lem12}\ref{p5}, we obtain $e'=\frac{s+4}{2}$.
Since $d(u)=\frac{2m-4}{3}$ is even for $m=3k+2$, $s$ is also even, by $d(u)=2t+s$. Moreover, $e'\geq s$. Together with $e'=\frac{s+4}{2}$, this implies that $(s,e')=(0,2)$, $(2,3)$ or $(4,4)$. Since $G^*[\overline{N}]=2K_1$ and $\delta(G^*)\geq2$, each of $w_1$ and $w_2$ must have at least two neighbors in $N$. Hence $e'\geq4$, and therefore the pairs $(0,2)$ and $(2,3)$ are impossible. Thus, it remains to consider the pair $(s,e')=(4,4)$. In this case, each of $w_1$ and $w_2$ has exactly two neighbors in $N$, in particular, all four edges between $N$ and $\overline{N}$ are incident with the four isolated vertices of $G^*[N]$. Up to isomorphism, this yields the unique possible configuration $G_5$, shown in Figure~\ref{fi4}.
\begin{figure}[ht]
    \centering
    \includegraphics[scale=.4]{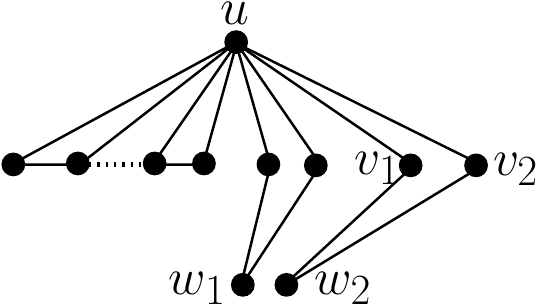}
    \caption{Graph $G_5$.}
    \label{fi4}
\end{figure}
Suppose that $G^*\cong G_5$, and let the vertices $u,w_1,w_2,v_1,v_2$ be as indicated in Figure~\ref{fi4}. Without loss of generality, assume that $x_{w_1}\geq x_{w_2}$. Define $G^{*{\prime}}=G^* -\{v_1 w_2, v_2 w_2\}+\{v_1 w_1, v_2 w_1\}$. By Lemma \ref{lem3}, $q(G^{*{\prime}}) > q(G^*)$. In $G^{*\prime}$, the vertex $w_2$ is isolated. Delete this isolated vertex and denote the resulting graph by $G^{*{\prime\prime}}=G^{*{\prime}} - {w_2}$. Then $G^{*\prime\prime}\in\mathcal{G}$ and $q(G^{*\prime\prime})=q(G^{*\prime})>q(G^*)$, contradicting the maximality of $q(G^*)$ in $\mathcal{G}$. Therefore $G^*\not\cong G_5$.

\textbf{Possibility 2.} $G^*[\overline{N}]=K_2$. Since $e(\overline{N})=1$, Lemma~\ref{lem12}\ref{p5} gives $e'=\frac{s+2}{2}$. As above, $s$ is even. Together with $e'\geq s$, this implies that $(s,e')=(0,1)$ or $(2,2)$. The pair $(0,1)$ is impossible, since $\delta(G^*)\geq2$ and $G^*[\overline{N}]=K_2$ imply that at least two edges join $N$ and $\overline{N}$. For the pair $(2,2)$, the two isolated vertices of $G^*[N]$ must be adjacent to the two vertices of $\overline{N}$, respectively. This results in $G^*\cong H_{\frac{2m-7}{3}}^{(5)}$,
contradicting the assumption $G^*\not\cong H_{\frac{2m-7}{3}}^{(5)}$.

Since both Possibilities~1 and~2 lead to contradictions, Subcase~2.3 cannot arise, and hence neither can Case~2.
\medskip\\
\noindent \textbf{Case 3.} $d(u)=\frac{2m-1}{3}-2=\frac{2m-7}{3}$. First suppose that $e'=s$. Then Lemma~\ref{lem5} gives $q(G^*)\leq \frac{2m-1}{3}$, contrary to \eqref{eq11}. Hence $e'>s$. From Lemma~\ref{lem7}, we obtain $|\overline{N}|\leq \frac72$. Hence $|\overline{N}|=0,1,2$ or $3$. Now, if $|\overline{N}|=0$, then $e'=0$. Since $s\geq0$, we cannot have $e'>s$. If $|\overline{N}|=1$, then, as in Subcase~1.2 of the proof of Theorem~\ref{r1}, we necessarily have $e'=s$, again a contradiction. Thus the remaining possibilities are $|\overline{N}|=2$ and $|\overline{N}|=3$.
\medskip\\
\noindent\textbf{Subcase 3.1.} $|\overline{N}|=2$. Let $\overline{N}=\{w_1,w_2\}$. There are two possibilities for $G^*[\overline{N}]$.

\textbf{Possibility 1.} $G^*[\overline{N}]=2K_1$. Then $e(\overline{N})=0$, and Lemma~\ref{lem12}\ref{p5} gives $e'=\frac{s+7}{2}$. Since $d(u)=\frac{2m-7}{3}$ is odd, $s$ is also odd. Together with $e'>s$, this implies that $(s,e')=(1,4),(3,5)$ or $(5,6)$. The pair $(1,4)$ is impossible. Indeed, by the same $\{\theta (1, 2, 2), \theta (1, 2, 3)\}$-free restriction used earlier, all edges between $N$ and $\overline{N}$ must be incident with isolated vertices of $G^*[N]$. However, when $s=1$, there is only one such vertex, and since $G^*$ is simple, it can be adjacent to at most the two vertices $w_1$ and $w_2$. Hence $e'\leq2$, contradicting $e'=4$. It now remains to consider the pairs $(3,5)$ and $(5,6)$.

\textbf{(i)} For the pair $(3,5)$, by the same $\{\theta (1, 2, 2), \theta (1, 2, 3)\}$-free restriction, all five edges between $N$ and $\overline{N}$ are incident with the three isolated vertices of $G^*[N]$. Moreover, since $G^*[\overline{N}]=2K_1$ and $\delta(G^*)\geq2$, each of $w_1$ and $w_2$ has at least two neighbors in $N$. These conditions determine, up to isomorphism, a unique graph in $\mathcal G$, denoted by $G_6$ and shown in Figure~\ref{fi5}.
\begin{figure}[ht]
    \centering
    \includegraphics[scale=.4]{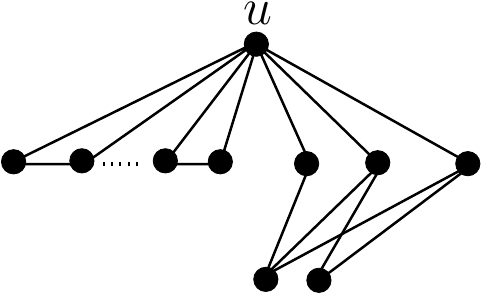}
    \caption{Graph $G_6$.}
    \label{fi5}
\end{figure}
If $G^*\cong G_6$, $Q(G^*)$ admits an equitable partition corresponding to which the quotient matrix is
\begin{align*}
    E_{G^*} =
    \left(\begin{smallmatrix}
        \frac{2m-7}{3} & \frac{2m-16}{3} & 1 & 2 & 0 & 0\\
        1 & 3 & 0 & 0 & 0 & 0\\
        1 & 0 & 2 & 0 & 1 & 0\\
        1 & 0 & 0 & 3 & 1 & 1\\
        0 & 0 & 1 & 2 & 3 & 0\\
        0 & 0 & 0 & 2 & 0 & 2
    \end{smallmatrix}\right).
\end{align*}
Its characteristic polynomial is $g(x) =  x^6-\frac{2}{3}(m+16)x^5 + (8m+34)x^4
    -\frac{8}{3}(13m+4)x^3  + \frac{1}{3}(202m-401)x^2 -\frac{1}{3}(172m-728)x +16m-128$. If $\lambda_g$ is the largest root of $g(x)=0$, then $q(G^*) = \lambda_g$ by Lemma \ref{lem4}. We show that $g(x)>0$ for all $x\geq \frac{2m-1}{3}$. Differentiating successively with respect to $x$, we obtain $g'(x) =  6x^5-\frac{10}{3}(m+16)x^4 + 4(8m+34)x^3 -8(13m+4)x^2  + \frac{2}{3}(202m-401)x-\frac{1}{3}(172m-728)$, $g''(x) =  30x^4-\frac{40}{3}(m+16)x^3 + 12(8m+34)x^2 -16(13m+4)x  + \frac{2}{3}(202m-401)$, $g^{(3)}(x) =  120x^3-40(m+16)x^2 + 24(8m+34)x -16(13m+4)$, $g^{(4)}(x) =  360x^2-80(m+16)x + 24(8m+34)$, $g^{(5)}(x) =  720x-80(m+16)$ and $g^{(6)}(x) =  720$. Since $g^{(6)}(x)>0$, $g^{(5)}(x)$ is increasing. For $x\geq \frac{2m-1}{3}$, $g^{(5)}(x) \geq g^{(5)}\!\left(\frac{2m-1}{3}\right) = 400m - 1520 > 0$ as  $m \geq 20$. Hence $g^{(4)}(x)$ is increasing on this interval. Moreover, $g^{(4)}(x) \geq g^{(4)}\!\left(\frac{2m-1}{3}\right) = \frac{1}{3}(320 m^2 - 2384m + 3848)>0$ for $m\geq 20$, Thus $g^{(3)}(x)$ is increasing for $x\geq \frac{2m-1}{3}$. At $x=\frac{2m-1}{3}$, we have $g^{(3)}(x) = \frac{1}{9}(160m^3- 1728m^2 + 5208m - 3704)\geq 76584 > 0$, $g''(x) = \frac{1}{81}(160m^4- 2144m^3+ 8160m^2 - 4796m - 15584)\geq 143216 > 0$ and $g'(x) = \frac{1}{243}(32m^5-416m^4 +464m^3+ 12632m^2 - 59114m + 78368)\geq 179016 > 0$ for all $m\geq 20$. Each of these polynomials is increasing in $m$ for $m\geq 20$, which can be verified by differentiation. Hence each attains its minimum at $m=20$ over the required range. Consequently, $g'(x)>0$ for $x\geq \frac{2m-1}{3}$, and therefore $g(x)$ is increasing on this interval. Finally, $g\left(\tfrac{2m-1}{3}\right) = \frac{1}{243}(32m^5-752m^4+6680m^3-29020m^2+62932m-54160)\ge 103360 >0$ for $m\geq 20$, by the similar procedure. Therefore, $g(x)\geq g(\tfrac{2m-1}{3})>0$ for all $x\geq \frac{2m-1}{3}$, and from this, $\lambda_g<\frac{2m-1}{3}$ follows. Thus $q(G^*)<\frac{2m-1}{3}$, which is a contradiction according to \eqref{eq11}. Theredore $G^*\not\cong G_6$.

\textbf{(ii)} For the pair $(5,6)$, all six edges between $N$ and $\overline{N}$ are incident with the five isolated vertices of $G^*[N]$. The minimum degree condition on $w_1$ and $w_2$, together with $\{\theta (1, 2, 2), \theta (1, 2, 3)\}$-free restriction, determines, up to isomorphism, exactly two possible configurations in $\mathcal{G}$, denoted by $G_7$ and $G_8$ and illustrated in Figure~\ref{fi6}.

\begin{figure}[ht]
    \centering
    \includegraphics[scale=.4]{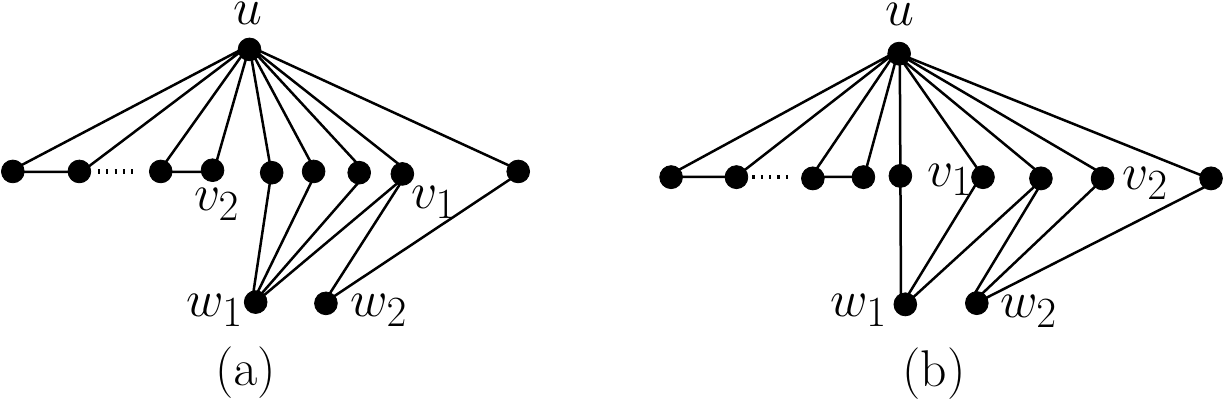}
    \caption{The graphs (a) $G_7$ and (b) $G_8$.}
    \label{fi6}
\end{figure}

Suppose first that $G^*\cong G_7$, with vertices $u,w_1,w_2,v_1,v_2$ labelled as in Figure~\ref{fi6}(a). If $x_u\geq x_{v_1}$, we define $G^{*{\prime}}=G^* -\{v_1w_2\}+\{u w_2\}$. Then $G^{*{\prime}} \in \mathcal{G}$, and Lemma \ref{lem3} gives $q(G^{\prime}) > q(G^*)$, contradicting the maximality of $q(G^*)$. If $x_{u} < x_{v_1}$, define $G^{*\prime\prime}=G^* -\{uv_2\}+\{v_1v_2\}$. Then $G^{*\prime\prime} \in \mathcal{G}$, and Lemma \ref{lem3} yields $q(G^{*\prime\prime}) > q(G^*)$, a contradiction again. Hence $G^*\not\cong G_7$.

Next suppose that $G^*\cong G_8$, with vertices labelled as in Figure~\ref{fi6}(b). Without loss of generality, assume that $x_{w_1}\geq x_{w_2}$. Define $G^{*{\prime}}=G^* -\{w_2v_2\}+\{w_1v_2\}$. Then $G^{\prime} \in \mathcal{G}$, and Lemma \ref{lem3} implies $q(G^{\prime}) > q(G^*)$, again contradicting the maximality of $q(G^*)$. Hence $G^*\not\cong G_8$. Therefore Possibility~1 cannot occur.

\textbf{Possibility 2.} $G^*[\overline{N}]=K_2$. Then $e(\overline{N})=1$, and Lemma~\ref{lem12}\ref{p5} gives $e'=\frac{s+5}{2}$. Since $e'>s$ and $s$ is odd, we obtain $(s,e')=(1,3)\text{ or }(3,4)$. The pair $(1,3)$ yields no admissible graph in $\mathcal{G}$ under the minimum degree condition and $\{\theta(1,2,2),\theta(1,2,3)\}$-free restriction. For the pair $(3,4)$, these restrictions determine, up to isomorphism, a unique graph in $\mathcal{G}$, denoted by $G_9$ and shown in Figure~\ref{fi7}(a).
\begin{figure}[ht]
    \centering
    \includegraphics[scale=.4]{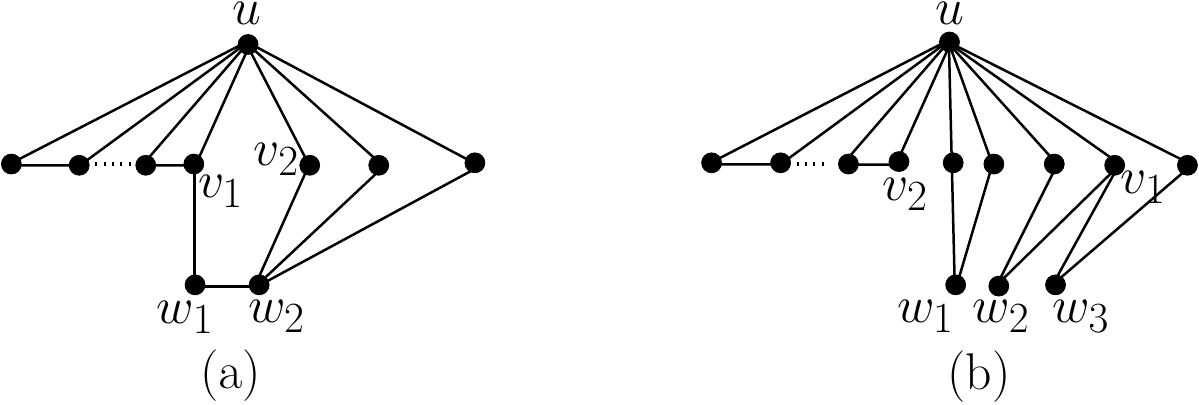}
    \caption{The graphs (a) $G_9$ and (b) $G_{10}$.}
    \label{fi7}
\end{figure}
Suppose that $G^*\cong G_9$, with vertices $u,w_1,w_2,v_1,v_2$ labeled as in Figure~\ref{fi7}(a). If $x_u\geq x_{v_1}$, let $G^{*\prime}=G^*-\{v_1w_1\}+\{uw_1\}$. Here $d_{G^*}(v_1)=3$, so $d_{G^{*\prime}}(v_1)=2$, while the degree of $w_1$ is unchanged; thus no pendant vertex is created. Moreover, since $G^*$ is $\{\theta(1,2,2),\theta(1,2,3)\}$-free, any forbidden theta graph newly created in $G^{*\prime}$ must contain the new edge $uw_1$. As is clear from Figure~\ref{fi7}(a), the edge $uw_1$ lies in no triangle in $G^{*\prime}$, and hence cannot belong to a $\theta(1,2,2)$. Further, every $4$-cycle in $G^{*\prime}$ containing $uw_1$ is of the form $uw_1w_2zu$, where $z\in N$ is adjacent to $w_2$, and none of the edges of such a $4$-cycle lies in a triangle. Hence $uw_1$ cannot belong to a newly created $\theta(1,2,3)$ either. Therefore $G^{*\prime}$ remains $\{\theta(1,2,2),\theta(1,2,3)\}$-free. The size and connectedness are also preserved, and therefore $G^{*\prime}\in\mathcal G$. By Lemma~\ref{lem3}, $q(G^{*\prime})>q(G^*)$, a contradiction. If $x_u<x_{v_1}$, define $G^{*\prime\prime}=G^*-\{uv_2\}+\{v_1v_2\}$. Here $v_2$ loses one edge and gains one edge, so its degree is unchanged and again no pendant vertex is created. Any newly created forbidden theta would have to contain the new edge $v_1v_2$. This edge lies in no triangle in $G^{*\prime\prime}$, while the only $4$-cycle containing it is $v_1v_2w_2w_1v_1$, none of whose edges lies in a triangle. Thus $G^{*\prime\prime}$ is also $\{\theta(1,2,2),\theta(1,2,3)\}$-free. Since the size and connectedness are preserved, $G^{*\prime\prime}\in\mathcal G$. Lemma~\ref{lem3} then gives $q(G^{*\prime\prime})>q(G^*)$, again a contradiction. Thus $G^*$ cannot be isomorphic to $G_9$. Hence Possibility~2 cannot occur, and therefore neither can Subcase~3.1.
\medskip\\
\noindent \textbf{Subcase 3.2.} $|\overline{N}|=3$. Let $\overline{N}=\{w_1,w_2,w_3\}$.

\textbf{Possibility 1.} $G^*[\overline{N}]=3K_1$. Then $e(\overline{N})=0$. Lemma~\ref{lem12}\ref{p5} gives $e'=\frac{s+7}{2}$, which yields $(s,e')=(1,4),(3,5)$ or $(5,6)$.
Since $G^*[\overline{N}]=3K_1$ and $\delta(G^*)\geq2$, each vertex of $\overline{N}$ has at least two neighbors in $N$. Hence $e'\geq6$, and therefore the pairs $(1,4)$ and $(3,5)$ are impossible. For the pair $(5,6)$, the minimum degree condition, together with the $\{\theta(1,2,2),\theta(1,2,3)\}$-free restriction, determines, up to isomorphism, a unique graph in $\mathcal{G}$, denoted by $G_{10}$ and shown in Figure~\ref{fi7}(b). Suppose that $G^*\cong G_{10}$, and let the vertices $u,w_1,w_2,w_3,v_1,v_2$ be as indicated in Figure~\ref{fi7}(b). If $x_u\geq x_{v_1}$, we define $G^{*{\prime}}=G^* -\{v_1w_3\}+\{u w_3\}$. Then $G^{*{\prime}} \in \mathcal{G}$, and applying Lemma \ref{lem3}, we get $q(G^{*{\prime}}) > q(G^*)$, a contradiction. If $x_{u} < x_{v_1}$, we define $G^{*\prime\prime}=G^* -\{u v_2\}+\{v_1v_2\}$. Then $G^{*\prime\prime} \in \mathcal{G}$, and Lemma \ref{lem3} implies $q(G^{*\prime\prime}) > q(G^*)$, a contradiction again. Thus $G^*$ cannot be isomorphic to $G_{10}$, and hence Possibility~1 cannot occur.

\textbf{Possibility 2.} $G^*[\overline{N}]=K_2\cup K_1$. Then $e(\overline{N})=1$. Lemma~\ref{lem12}\ref{p5} gives
$e'=\frac{s+5}{2}$, which implies $(s,e')=(1,3)$ or $(3,4)$. The pair $(1,3)$ is impossible. Indeed, since $\delta(G^*)\geq2$, the isolated vertex of $G^*[\overline{N}]$ has at least two neighbors in $N$, while each endpoint of the $K_2$ has at least one neighbor in $N$. Hence $e'\geq4$, contradicting $e'=3$. For the pair $(3,4)$, it determines, up to isomorphism, exactly two graphs in $\mathcal{G}$, denoted by $G_{11}$ and $G_{12}$ and shown in Figure~\ref{fi9}.
\begin{figure}[ht]
    \centering
    \includegraphics[scale=.4]{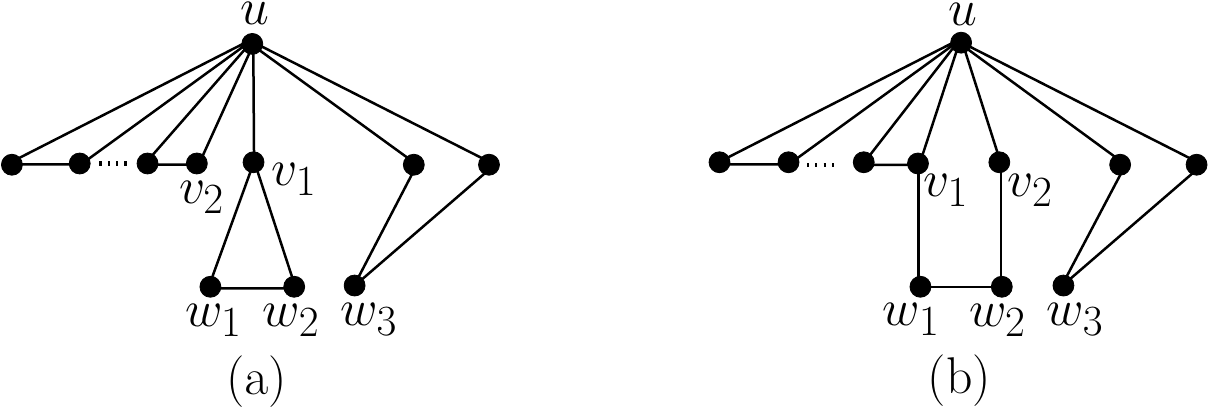}
    \caption{The graphs (a) $G_{11}$ and (b) $G_{12}$.}
    \label{fi9}
\end{figure}
Suppose that $G^*\cong G_i$ for some $i\in\{11,12\}$, and let the vertices $u,w_1,w_2,w_3,v_1,v_2$ be as indicated in the graphs in Figure~\ref{fi9}. If $x_u\geq x_{v_1}$, we define $G^{*{\prime}}=G^* -\{v_1w_1\}+\{u w_1\}$. Then $G^{*{\prime}} \in \mathcal{G}$, and Lemma \ref{lem3} implies that $q(G^{*{\prime}}) > q(G^*)$, a contradiction. If $x_{u} < x_{v_1}$, we define $G^{*\prime\prime}=G^* -\{u v_2\}+\{v_1 v_2\}$. Then $G^{*\prime\prime} \in \mathcal{G}$, and Lemma \ref{lem3} gives $q(G^{*\prime\prime}) > q(G^*)$, a contradiction again. Thus $G^*$ cannot be isomorphic to either $G_{11}$ or $G_{12}$, and hence Possibility~2 cannot occur.

\textbf{Possibility 3.} $G^*[\overline{N}]=P_3$. Since $e(\overline{N})=2$, Lemma~\ref{lem12}\ref{p5}, together with $e'>s$, gives $(s,e')=(1,2)$. Corresponding to this pair, up to isomorphism, we obtain exactly two graphs in $\mathcal{G}$, denoted by $G_{13}$ and $G_{14}$ and shown in Figure~\ref{fi10}.
\begin{figure}[ht]
    \centering
    \includegraphics[scale=.4]{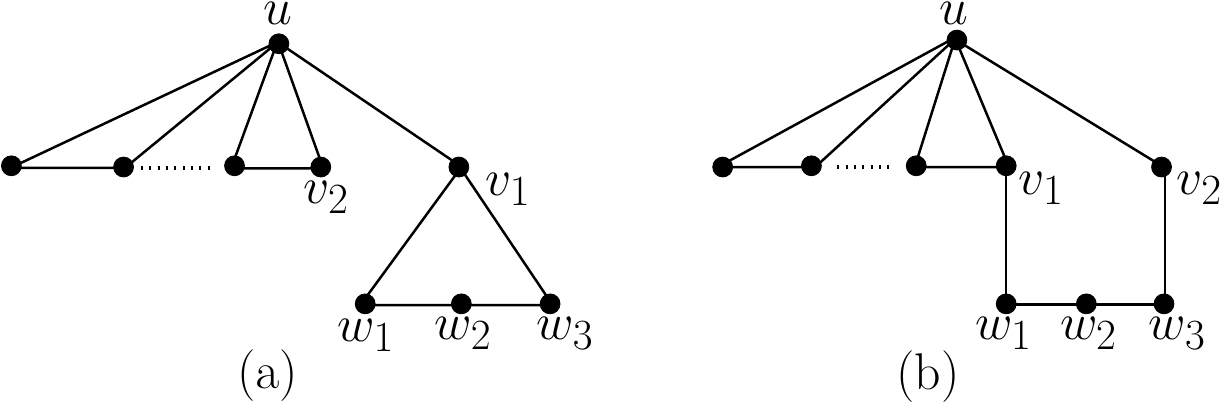}
    \caption{The graphs (a) $G_{13}$ and (b) $G_{14}$.}
    \label{fi10}
\end{figure}
Suppose that $G^*\cong G_i$ for some $i\in\{13,14\}$, and let the vertices $u,w_1,w_2,w_3,v_1,v_2$ be as indicated in each graphs in Figure~\ref{fi10}. If $x_u\geq x_{v_1}$, define $G^{*{\prime}}=G^* -\{v_1w_1\}+\{u w_1\}$. Then $G^{*{\prime}} \in \mathcal{G}$, and applying Lemma \ref{lem3}, we get $q(G^{*{\prime}}) > q(G^*)$, a contradiction. If $x_{u} < x_{v_1}$, we define $G^{*\prime\prime}=G^* -\{u v_2\}+\{v_1v_2\}$. Then $G^{*\prime\prime} \in \mathcal{G}$, and Lemma \ref{lem3} gives $q(G^{*\prime\prime}) > q(G^*)$, again a contradiction. Thus $G^*$ cannot be isomorphic to either $G_{13}$ or $G_{14}$, and hence Possibility~3 cannot occur.

\textbf{Possibility 4.} $G^*[\overline{N}]=K_3$. Since $e(\overline{N})=3$, Lemma~\ref{lem12}\ref{p5} gives $e'=\frac{s+1}{2}$. Together with $e'>s$, this implies $s<1$, which is impossible since $s$ is a nonnegative odd integer. Hence Possibility~4 cannot occur. Since all possibilities in Subcase~3.2 lead to contradictions, Subcase~3.2, and hence Case~3, cannot occur.
\medskip\\
\noindent \textbf{Case 4.} $2\leq d(u)\leq \frac{2m-1}{3}-3=\frac{2m-10}{3}$. As in Case~3 of the proof of Theorem~\ref{r1}, Lemma~\ref{lem2} gives
\begin{align}\label{eq23}
    q(G^*)\leq d(u)+\frac{m}{d(u)}+\frac{1}{2}.
\end{align}
For $2\leq x\leq \frac{2m-10}{3}$, we define $F(x)=\frac{2m-1}{3}-\left(x+\frac{m}{x}+\frac{1}{2}\right)$. Since $F''(x)=-\frac{2m}{x^3}<0$, the function $F$ is concave on this interval. Hence $F(x)\geq \min\left\{F(2),F\left(\frac{2m-10}{3}\right)\right\}$. Now, $F(2) = \frac{2m-1}{3}-2-\frac{m}{2}-\frac{1}{2}=\frac{m-17}{6}> 0$ and $F\left(\frac{2m-10}{3}\right) = \frac{2m-1}{3}-\frac{2m-10}{3}-\frac{3m}{2m-10}-\frac{1}{2}=\frac{2m-25}{2m-10}> 0$, for $m\geq 20$. Therefore, $F(x)>0$ throughout the interval, and taking $x=d(u)$ gives $d(u)+\frac{m}{d(u)}+\frac{1}{2}<\frac{2m-1}{3}$. Together with \eqref{eq23}, this yields $q(G^*)<\frac{2m-1}{3}$, contrary to \eqref{eq11}. Hence Case~4 cannot occur.

We have shown that, under the assumption $G^*\not\cong H_{\frac{2m-7}{3}}^{(5)}$, none of Cases~1--4 can occur. Therefore, $G^*\cong H_{\frac{2m-7}{3}}^{(5)}$. Furthermore, from the polynomial $h(x)$ obtained in the proof of Lemma~\ref{lem13}, the maximum signless Laplacian spectral radius is the largest root of $x^4-\frac{2}{3}(m+10)x^3
+\frac{2}{3}(7m+16)x^2
-(10m-7)x
+\frac{4}{3}(5m-16)=0$. This proves Theorem~\ref{r2}.\qedhere
\end{proof}


\section{Concluding remarks}\label{sec5}

In this paper, we characterized the unique graph attaining the maximum signless Laplacian spectral radius in two settings: the fixed-order problem for $\{C_3,C_4\}$-free graphs of sufficiently large order, and the fixed-size problem for $\{\theta(1,2,2),\theta(1,2,3)\}$-free graphs with sufficiently large size; both under the condition of having no pendant vertices. Together with the result of Liu and Wang~\cite{maxima_Q_delta_by_liu_wang} for $m=3k$ $(k\geq 3)$, our fixed-size results complete the characterization for all sufficiently large values of $m$ in that setting.

Natural continuations of this work include removing the restriction on the absence of pendant vertices and determining the corresponding extremal graphs in this more general setting, which may exhibit substantially different structural behavior. It would also be interesting to study analogous problems for the $A_\alpha$-spectral radius, $\alpha\in[0,1)$, thereby extending the present results to a wider spectral framework.


\section*{Statements and Declarations} 

\textbf{Competing interests:} There are no competing interests.\par\noindent
\textbf{Data availability:} No datasets were generated or analyzed during the current study.\par\noindent
\textbf{Funding information:} This research received no external funding.


\bibliographystyle{plain}
\bibliography{biblio}
\end{document}